\documentclass[12pt]{amsart}
\usepackage[latin9]{inputenc}
\usepackage{color}
\usepackage{verbatim}
\usepackage{amsthm}
\usepackage{amssymb}
\usepackage{multicol}
\usepackage{tabularx}
\usepackage{graphicx}

\makeatletter
\numberwithin{equation}{section}
\numberwithin{figure}{section}
\usepackage{enumitem}		
\theoremstyle{plain}
\newtheorem{thm}{\protect\theoremname}[section]
  \theoremstyle{definition}
  \newtheorem{example}[thm]{\protect\examplename}
  \theoremstyle{remark}
  \newtheorem{claim}[thm]{\protect\claimname}
    \theoremstyle{remark}
    \newtheorem{question}[thm]{\protect\questionname}
  
  \theoremstyle{plain}
  \newtheorem{lem}[thm]{\protect\lemmaname}
  \theoremstyle{plain}
  \newtheorem{conjecture}[thm]{\protect\conjecturename}
  \theoremstyle{plain}
  \newtheorem{cor}[thm]{\protect\corollaryname}

\usepackage{fullpage}
\usepackage{tikz}
\usepackage{pgf}

\usepackage{caption}
\usepackage[labelformat=simple,labelfont={}]{subcaption}

\usetikzlibrary{shapes}
\usepackage{graphicx}
\tikzstyle{vertex}=[shape=circle,inner sep=0,fill=black,minimum size=1.5mm,draw]

\makeatother

  \providecommand{\claimname}{Claim}
    \providecommand{\questionname}{Question}
  \providecommand{\conjecturename}{Conjecture}
  \providecommand{\corollaryname}{Corollary}
  \providecommand{\examplename}{Example}
  \providecommand{\lemmaname}{Lemma}
\providecommand{\theoremname}{Theorem}

\begin{document}



\title{On the Corr\' adi-Hajnal Theorem and a question of Dirac}\thanks{The first two authors
thank Institut Mittag-Leffler (Djursholm, Sweden) for the hospitality and creative environment.}

\author{H.A. Kierstead}

\thanks{Department of Mathematics and Statistics, Arizona State University,
Tempe, AZ 85287, USA. E-mail address: kierstead@asu.edu. Research
of this author is supported in part by NSA grant H98230-12-1-0212.}

\author{A.V. Kostochka}

\thanks{Department of Mathematics, University of Illinois, Urbana, IL, 61801,
USA and Institute of Mathematics, Novosibirsk, Russia. E-mail address:
kostochk@math.uiuc.edu. Research of this author is supported in part
by NSF grants DMS-0965587    and   DMS-1266016 and by grant 12-01-00448 
of the Russian Foundation for Basic Research. }

\author{E.C. Yeager}

\thanks{Corresponding author. Department of Mathematics, University of Illinois, Urbana, IL, 61801,
USA. E-mail address: yeager2@illinois.edu. Phone: +1 (304) 276-9647. Research of this author is supported in part
by NSF grant DMS-1266016.}
\begin{abstract}
In 1963, Corr\' adi and Hajnal proved that for all $k\geq1$ and $n\geq3k$,
every graph $G$ on $n$ vertices with minimum degree $\delta(G)\geq2k$
contains $k$ disjoint cycles. The bound $\delta(G) \geq 2k$ is sharp. Here we
characterize those graphs with $\delta(G)\geq2k-1$ that contain $k$
disjoint cycles. This answers the simple-graph case of Dirac's 1963 question on the characterization
of $(2k-1)$-connected graphs with no $k$ disjoint cycles.

Enomoto and Wang refined the Corr\'adi-Hajnal Theorem, proving the following Ore-type version:
For all $k\geq1$ and $n\geq3k$, every graph $G$ on $n$
vertices contains $k$ disjoint cycles, provided that $d(x)+d(y)\geq4k-1$ for all distinct nonadjacent vertices $x,y$. We refine this further for
$k\geq3$ and $n\geq3k+1$: If $G$ is a graph on $n$ vertices such that $d(x)+d(y)\geq4k-3$ 
for all distinct nonadjacent vertices $x,y$,
then $G$ has $k$ vertex-disjoint
cycles if and only if the independence number $\alpha(G)\leq n-2k$  and $G$ is not one of two small exceptions in the case $k=3$.
We also show how the case $k=2$ follows from Lov\' asz' characterization of multigraphs with no two disjoint cycles.
\end{abstract}
\maketitle
{\small{Mathematics Subject Classification: 05C15, 05C35.}}{\small \par}

{\small{Keywords: Disjoint cycles, Ore-degree, graph packing, equitable
coloring, minimum degree.}}{\small \par}

\section{Introduction}

For a graph $G=(V,E)$, let $|G|=|V|$, $\|G\|=|E|$, $\delta(G)$ be the minimum degree of  $G$, and $\alpha(G)$ 
be the independence number of $G$. Let $\overline G$ denote the complement of $G$ and 
for disjoint graphs $G$ and $H$, let $G\vee H$ denote $G\cup H$ together with all edges from $V(G)$ to $V(H)$. The degree of a vertex $v$ in a graph $H$ is $d_H(v)$; when $H$ is clear, we write $d(v)$.

In 1963, Corr\' adi and Hajnal proved a conjecture of Erd\H os by showing the following:
\begin{thm}[\cite{CH}]
 \label{cht} Let $k\in\mathbb{Z}^{+}$. Every graph $G$ with (i)
$|G|\geq3k$ and (ii) $\delta(G)\geq2k$ contains $k$ disjoint cycles. 
\end{thm}
Clearly, hypothesis (i) in the theorem is sharp. Hypothesis (ii) also is sharp.
Indeed, if a graph $G$  has $k$ disjoint cycles, then $\alpha(G)\leq|G|-2k$, 
since every cycle contains at least
two vertices of $G-I$ for any independent set $I$. 
 Thus  $H:=\overline{K_{k+1}}\vee K_{2k-1}$
satisfies (i) and has $\delta(H)=2k-1$, but  does not have $k$ disjoint
cycles, because $\alpha(H)=k+1>|H|-2k$. There are several works refining Theorem~\ref{cht}.
Dirac and Erd\H os~\cite{DE} showed that if a graph $G$ has many more vertices of degree at least $2k$ than vertices of
degree at most $2k-2$, then $G$ has $k$ disjoint cycles. Dirac~\cite{Di} asked:

\begin{question}\label{DiracQ}
 Which $(2k-1)$-connected graphs do not have $k$ disjoint cycles?
 \end{question} 
 
  He also resolved his question for $k=2$ by describing all $3$-connected multigraphs on at least $4$ vertices in which every two
  cycles intersect. It turns out that the only simple $3$-connected graphs with this property are wheels.
  Lov\' asz~\cite{Lo} fully described all multigraphs in which every two
  cycles intersect.

 The following result in this paper yields a full answer to Dirac's question for simple graphs. 
\begin{thm}
\label{ch++}Let $k\geq 2$. Every graph $G$ with (i)~$|G|\geq3k$
and (ii)~$\delta(G)\geq2k-1$ contains $k$ disjoint cycles if and
only if
\begin{enumerate}[label=(H\arabic*), ref=(H\arabic*)]
\item[(H3)]~$\alpha(G)\leq|G|-2k$, and
\item[(H4)]~if $k$ is odd and $|G|=3k$, then $G\neq 2K_k \vee\overline{K_k} $
 and if $k=2$ then $G$ is not a wheel.
\end{enumerate}
\end{thm}
Since for every independent set $I$ in a graph $G$ and every $v\in I$, $N(v)\subseteq V(G)-I$,
if $\delta(G)\geq2k-1$ and $|I|\geq |G|-2k+1$, then $|I|=|G|-2k+1$ and $N(v)=V(G)-I$ for every
$v\in I$. It follows that every graph $G$ satisfying (ii) and not satisfying~(H3) contains $K_{2k-1,|G|-2k+1}$ 
and is contained in $K_{|G|}-E(K_{|G|-2k+1})$.
The conditions of Theorem~\ref{ch++} can be tested in polynomial
time.

Most likely, Dirac intended his question to refer to multigraphs; indeed, his result for $k=2$ 
is for
 multigraphs. But the case of simple graphs is the most important in the question.
 In~\cite{KKY2} we heavily use the results of this paper to obtain
 a characterization of $(2k-1)$-connected multigraphs that contain $k$ disjoint cycles, answering Question~\ref{DiracQ} in full.


Studying Hamiltonian properties of graphs, Ore introduced the {\em minimum Ore-degree}  $\sigma_{2}$: If $G$ is a complete graph, then
 $\sigma_{2}(G)=\infty$, otherwise $\sigma_2(G):=\min\{d(x)+d(y):xy\not\in E(G)\}$.
Enomoto \cite{Enomoto} and Wang \cite{Wang} generalized the Corr\' adi-Hajnal
Theorem in terms of  $\sigma_{2}$: 
\begin{thm}[\cite{Enomoto},\cite{Wang}]
\label{enomoto} Let $k\in\mathbb{Z}^{+}$. Every graph $G$ with
(i) $|G|\geq3k$ and
\begin{enumerate}[label=(H\arabic*), ref=(H\arabic*)]
\item[(E2)] $\sigma_{2}(G)\geq4k-1$
 \end{enumerate}
  contains $k$ disjoint
cycles. 
\end{thm}
Again $H:=\overline{K_{k+1}}\vee K_{2k-1}$ shows that hypothesis (E2)
of Theorem~\ref{enomoto} is sharp. What happens if we relax (E2) to (H2): $\sigma_2(G)\ge4k-3$, but again add hypothesis (H3)?
 Here are two interesting examples.

\begin{example}
Let $k=3$ and  $\mathbf Y_1$ be the graph obtained by  twice subdividing one of the edges $wz$ of $K_8$, i.e., replacing $wz$ by the path $wxyz$. Then $|\mathbf Y_1|=10=3k+1$, $\sigma_2(\mathbf Y_1)=9=4k-3$, and $\alpha(\mathbf Y_1)=2\le|\mathbf Y_1|-2k$. However, $\mathbf Y_1$ does not contain $k=3$ disjoint cycles, since each cycle would need to contain three vertices of the original $K_8$ (see Figure~\ref{Y1}).
\end{example}

\begin{figure}[ht]\label{K8 subdivided}
\begin{center}
\begin{subfigure}[b]{0.3\textwidth}\centering
\begin{tikzpicture}[scale=0.4, rotate=90]
\foreach \y in {1,...,8}
	{\draw(0,0) +(22.5+45*\y: 3 cm)  node[vertex] (p\y){} ;}
\draw (p1)+(-.5,1.25)  node[vertex] (q1){} ;
\draw (p2)+(.5,1.25)  node[vertex] (q2){} ;
\draw[thick] (p1) -- (q1) -- (q2) -- (p2);
\foreach \x in {2,...,8}{
	\foreach \y in {\x,...,8}{
		\draw (p\x) -- (p\y);
		}}
\foreach \x in {3,...,8}{
	\draw (p1) -- (p\x);}
\draw (-3.5,-2) node{};
\end{tikzpicture}\caption{${\bf Y_1}$}\label{Y1}
\end{subfigure}
\begin{subfigure}[b]{0.3\textwidth}
\centering
  \begin{tikzpicture}[scale=0.45]
\foreach \x in {1,...,4}
{ \draw(2*\x,0) node[vertex] (x\x){} ;}
\foreach \y in {1,...,3}
{ \draw(1+2*\y,3) node[vertex] (y\y){} ;}
\foreach \x in {1,...,4}
	\foreach \y in {1,2,3}
		{
		\draw (x\x) -- (y\y);
		}
\draw(3,-2) node[vertex,
label=below:$u$,draw] (u){} ;
\draw(7,-2) node[vertex,
label=below:$v$,draw] (v){} ;
\draw (x1) -- (u) -- (x2);
\draw (x3) -- (v) -- (x4);
\draw (u)--(v);
\draw(11,1) node[vertex](x){};
\draw[gray] (u)--(x)--(v);
\foreach \x in {1,...,4}
	\draw[gray] (x)--(x\x);
\foreach \y in {1,...,3}
	\draw[gray] (x)--(y\y);
 \end{tikzpicture}\caption{${\bf Y_2}$}\label{Y2}
\end{subfigure}
\end{center}\caption{}
\end{figure}

\begin{example}
Let $k=3$. Let $Q$ be obtained from $K_{4,4}$ by replacing a vertex $v$ and its incident edges $vw,vx,vy,vz$ by new vertices $u,u'$ and edges $uu',uw,ux,u'y,u'z$; so $d(u)=3=d(u')$ and contracting $uu'$ in $Q$ yields $K_{4,4}$. Now set   $\mathbf Y_2:=K_1\vee Q$.
  Then $|\mathbf Y_2|=10=3k+1$, $\sigma_2(\mathbf Y_2)=9=4k-3$, and $\alpha(\mathbf Y_2)=4\le|\mathbf Y_2|-2k$. However, $\mathbf Y_2$ does not contain $k=3$ disjoint cycles, since each $3$-cycle  contains the only vertex of $K_1$ (see Figure~\ref{Y2}).
\end{example}

Our  main  result is:
\begin{thm}
\label{main} 
Let  $k\in\mathbb{Z}^{+}$ with $k\ge3$.  Every graph $G$ with
\begin{enumerate}[label=(H\arabic*), ref=(H\arabic*)]
\item\label{H1} $|G|\ge3k+1$, 
\item\label{H2} {$\sigma_{2}(G)\geq4k-3$}, 
and
\item\label{H3} $\alpha(G)\leq|G|-2k$
\end{enumerate}
 contains $k$ disjoint cycles, unless $k=3$ and 
 $G\in \{\mathbf Y_1,\mathbf Y_2\}$. Furthermore, for fixed $k$ there is a polynomial time algorithm that either produces $k$ disjoint cycles or demonstrates that one of the hypotheses fails.
\end{thm}

 Theorem~\ref{main} 
  is proved
in Section 2.  In Section 3 we discuss the case $k=2$. 
In Section 4 we discuss connections to equitable colorings and derive Theorem~\ref{ch++}  from Theorem~\ref{main} and known results.

Now we show examples demonstrating
the sharpness of hypothesis
(H2) that $\sigma(G)\ge 4k-3$, then discuss some unsolved problems, and then review our notation.

\begin{example}
\label{H} Let $k\geq3$, $Q=K_3$ and $G_k:=\overline{K_{2k-2}}\vee(\overline{K_{2k-3}}+Q)$. Then
 $|G_k|=4k-2\geq3k+1$, $\delta(G_k)=2k-2$ and $\alpha(G_k)=|G_k|-2k$. 
If $G_k$ contained $k$ disjoint cycles, then at least $4k-|G_k|=2$ would be
$3$-cycles; this is impossible, since any $3$-cycle in $G_k$ contains
an edge of  $Q$. This construction can be extended. Let $k=r+t$, where $k+3\le2r\le2k$, $Q'=K_{2t}$, and put $H=G_r\vee Q'$. Then
 $|H|=4r-2+2t=2k+2r-2\geq3k+1$, $\delta(H)=2r-2+2t=2k-2$ and $\alpha(H)=2r-2=|H|-2k$.  If $H$ contained $k$ disjoint cycles, then at least $4k-|H|=2t+2$ would be
$3$-cycles; this is impossible, since any $3$-cycle in $H$ contains
an edge of $Q$ or a vertex of $Q'$. 
\end{example}
There are several special examples for small $k$. The constructions of $\mathbf Y_1$ and $\mathbf Y_2$ can be extended to $k=4$ at the cost of lowering $\sigma_2$ to $4k-4$. Below is another small family of special examples.  The blow-up of $G$ by $H$ is denoted by $G[H]$; that is, $V(G[H])=V(G)\times V(H)$
and $(x,y)(x',y')\in E(G[H])$ if and only if $xx'\in E(G)$, or $x=x'$ and $yy'\in E(H)$.

\begin{example}
\label{G_4} For $k=4$, $G:=C_{5}[\overline{K_{3}}]$ satisfies $|G|=15\geq3k+1$,
$\delta(G)=2k-2$ and $\alpha(G)=6<|G|-2k$. 
Since $\mbox{girth}(G)=4$, we see that $G$ has at most $\frac{|G|}{4}<k$ disjoint
cycles. This example can be extended to $k=5,6$ as follows. Let $I=\overline{K_{2k-8}}$ and $H=G\vee I$. Then $|G|=2k+7\geq3k+1$,
$\delta=2k-2$ and $\alpha(G)=6<|G|-2k=7$. If $H$ has $k$ disjoint cycles then each of the at least $k-(2k-8)=8-k$ cycles that do not meet $I$ use $4$ vertices of $G$, and the other cycles use at least $2$ vertices of $G$. Then $15=|G|\ge2k+2(8-k)=16$, a contradiction.
\end{example}

\noindent\textbf{Unsolved problems.} 
1. For every fixed $k$, we know only a finite number of extremal examples. It would be very interesting
to describe all graphs $G$ with $\sigma_{2}(G)=4k-4$ that do not have $k$ disjoint cycles, but this most likely would need
new techniques and approaches. 

2. Recently, there were several results in the spirit of  the Corr\' adi-Hajnal
Theorem giving degree conditions on a graph $G$ sufficient for the existence in $G$ of $k$ disjoint copies of 
such subgraphs as chorded cycles~\cite{BFG,CFGL} and $\Theta$-graphs~\cite{CFKS}.
It could be that our techniques can help in similar problems.

3. One also may try to sharpen the above-mentioned theorem of Dirac and Erd\H os~\cite{DE}.\\

\noindent\textbf{Notation.} 
A \emph{bud} is a vertex with degree $0$ or $1$.
A vertex is \emph{high} if it has degree at least $2k-1$, and \emph{low} otherwise.
For vertex subsets $A,B$ of a graph $G=(V,E)$, let 
$$\left\Vert A,B\right\Vert :=\sum_{u\in A}|\{uv\in E(G):v\in B\}|.$$
Note $A$ and $B$ need not be disjoint. For example, $\|V,V\|=2\|G\|=2|E|$. We will abuse this notation
to a certain extent. If $A$ is a subgraph of $G$,
we write $\left\Vert A,B\right\Vert $ for $\left\Vert V(A),B\right\Vert $, and if $\mathcal A$ is a set of disjoint subgraphs, we write $\left\Vert \mathcal A,B\right\Vert $ for $\left\Vert \bigcup_{H\in \mathcal A}V(H),B\right\Vert $.
Similarly, for $u\in V(G)$, we write $\left\Vert u,B\right\Vert $
for $\left\Vert \{u\},B\right\Vert $. Formally, an edge $e=uv$ is the set $\{u,v\}$; we often write $\|e,A\|$ for $\|\{u,v\},A\|$.

If  $T$ is a tree or a directed cycle and  $u,v\in V(T)$ we write  $uTv$ for the unique subpath of $T$ with endpoints $u$ and $v$.  We also extend this: if $w\notin T$, but has exactly one neighbor $u\in T$, we write $wTv$ for $w(T+w+wu)v$. Finally, if $w$ has exactly two neighbors $u,v \in T$, we may write $wTw$ for the cycle $wuTvw$. 

\section{Proof of Theorem \ref{main}}

\noindent Suppose $G=(V,E)$ is an edge-maximal counterexample to
Theorem \ref{main}. That is, for some  $k\geq3$, (H1)--(H3) hold, and  $G$
does not contain $k$ disjoint cycles, but adding any edge $e\in E(\overline{G})$
to $G$ results in a graph with $k$ disjoint cycles. The edge $e$
will be in precisely one of these cycles, so $G$ contains $k-1$ disjoint
cycles, and at least three additional vertices. Choose a set $\mathcal{C}$
of disjoint cycles in $G$ so that: 
\begin{enumerate}[label=(O\arabic*), ref=(O\arabic*)]
\item \label{o1}$|\mathcal{C}|$ is maximized; 
\item \label{o2}subject to \ref{o1}, $\sum_{C\in\mathcal{C}}|C|$
is minimized; 
\item \label{o3}subject to \ref{o1} and \ref{o2}, the length of a longest
path $P$ in $R:=G-\bigcup\mathcal{C}$ is maximized;
\item \label{o4} subject to \ref{o1}, \ref{o2}, and \ref{o3}, $\left\Vert R\right\Vert $
is maximized.
\end{enumerate}
Call such a $\mathcal{C}$ an \emph{optimal} \emph{set}. We prove
in Subsection~\ref{path} that $R$ is a path, and in Subsection
\ref{three} that $|R|=3$. We develop the structure of $\mathcal{C}$
in Subsection~\ref{sub:Key-L}. Finally, in
Subsection~\ref{contradiction}, these results are used to prove Theorem~\ref{main}.

Our arguments will have the following form.  We will make a series of claims about our optimal set $\mathcal C$, and then show that if any part of a claim fails, then we could have improved $\mathcal C$ by replacing a sequence  $C_1,\dots,C_t\in\mathcal C$ of at most three cycles by  another sequence of cycles $C'_1,\dots,C'_{t'}$. Naturally, this modification may also change $R$ or $P$. We will express the contradiction by writing ``$C'_1,\dots,C'_t,R',P'$ beats $C_1,\dots,C_t,R,P$,'' and may drop $R'$ and $R$ or $P'$ and $P$ if they are not involved in the optimality criteria.

This proof implies a polynomial time algorithm. We start by adding enough extra edges---at most $3k$---to obtain from $G$ 
a graph with a set $\mathcal C$ of $k$ disjoint cycles. Then we remove the extra edges in $\mathcal C$  one at a time. After removing an extra edge, we calculate a new collection  $\mathcal C'$. This is accomplished by checking the series of claims, each in polynomial time. If a claim fails, we calculate a better collection (again in polynomial time) and restart the check, or discover an independent set of size greater than $|G|-2k$. As there can be at most $n^4$ improvements, corresponding to adjusting the four parameters \ref{o1}--\ref{o4},
this process ends in polynomial time.

We now make some simple observations. Recall that $|\mathcal{C}|=k-1$
and $R$ is acyclic. By \ref{o2} and our initial remarks, $|R|\ge3$.
Let  $a_{1}$ and $a_{2}$ be the endpoints of $P$. (Possibly,
$R$ is an independent set, and $a_{1}=a_{2}$.)

\begin{claim}
\label{clm:1}For all $w\in V(R)$ and $C\in\mathcal{C}$, if $\left\Vert w,C\right\Vert \ge2$
then $3\leq|C|\leq6-\left\Vert w,C\right\Vert $. In particular, (a)
$\left\Vert w,C\right\Vert \leq3$, (b) if $\left\Vert w,C\right\Vert =3$
then $|C|=3$, and (c) if $|C|=4$ then the two neighbors of $w$
in $C$ are nonadjacent. \end{claim}
\begin{proof}
Let $\overrightarrow{C}$ be a cyclic orientation of $C$. 
For distinct $u,v\in N(w)\cap C$, the cycles $wu\overrightarrow{C}vw$
and $wu\overleftarrow{{C}}vw$  have length at least $|C|$ by
\ref{o2}. Thus $2\left\Vert C\right\Vert \leq\|wu\overrightarrow{C}vw\|+\|wu\overleftarrow{{C}}vw\|=\left\Vert C\right\Vert +4,$
so $|C|\leq4$. Similarly, if $\left\Vert w,C\right\Vert \geq3$ then
$3\|C\|\le\|C\|+6$, and so $|C| =3$. \end{proof}
The next claim is a simple corollary of condition \ref{o2}.
\begin{claim}
\label{clm1.2}If $xy\in E(R)$ and $C\in\mathcal{C}$ with $|C|\geq4$
then $N(x)\cap N(y)\cap C=\emptyset$.
\end{claim}

\subsection{$\mathbf{R}$ is a path\label{path}}

Suppose $R$ is not a path. Let $L$ be the set of buds in $R$; then
$|L|\geq3$.
\begin{claim}
\label{clm:2}For all $C\in\mathcal{C}$, distinct  $x,y,z\in V(C)$, $i\in[2]$, 
and $u\in V(R-P)$:
\begin{enumerate}[label=(\alph{enumi}), ref=\alph{enumi}]
\item \label{clm:2.1} $\{ux,uy,a_{i}z\}\nsubseteq E$;
\item \label{clm:2.2}$\left\Vert \{u,a_{i}\},C\right\Vert \leq4$;
\item \label{clm:2.3} $\{a_{i}x,a_{i}y,a_{3-i}z,zu\}\nsubseteq E$ ; 
\item \label{clm:2.4}if $\left\Vert \{a_{1},a_{2}\},C\right\Vert \geq5$
then $\left\Vert u,C\right\Vert =0$;
\item \label{clm:2.5}$\left\Vert \{u,a_{i}\},R\right\Vert \geq1$; in particular
$\left\Vert a_{i},R\right\Vert =1$ and $|P|\geq2$;
\item \label{clm:2.6}$4-\left\Vert u,R\right\Vert \leq\left\Vert \{u,a_{i}\},C\right\Vert $
and $\left\Vert \{u,a_{i}\},D\right\Vert =4$ for at least $|\mathcal C|-\left\Vert u,R\right\Vert $
cycles $D\in\mathcal{C}$.
\end{enumerate}
\end{claim}
\begin{proof}
(\ref{clm:2.1}) Else $ux(C-z)yu,Pa_{i}z$ beats $C,P$
by \ref{o3} (see Figure \ref{subfig1.1}).\\
(\ref{clm:2.2}) Else $|C|=3$ by Claim~\ref{clm:1}.
Then there are distinct $p,q,r\in V(C)$ with $up,uq,a_{i}r\in E$,
contradicting (\ref{clm:2.1}).\\
(\ref{clm:2.3}) Else $a_{i}x(C-z)ya_{i},(P-a_{i})a_{3-i}zu$ beats $C,P$ by \ref{o3} (see Figure \ref{subfig1.2}).\\
(\ref{clm:2.4}) Suppose
$\left\Vert \{a_{1},a_{2}\},C\right\Vert \geq5$ and $p\in N(u)\cap C$.
By Claim~\ref{clm:1}, $|C|=3$. Pick $j\in[2]$ with $pa_j\in E$, preferring $\|a_j,C\|=2$. 
Then
 $V(C)-p\subseteq N(a_{3-j})$, contradicting (\ref{clm:2.3}).\\
(\ref{clm:2.5}) Since $a_{i}$ is an end of the maximal path $P$, we get
$N(a_{i})\cap R\subseteq P$; so $a_{i}u\notin E$. By (\ref{clm:2.2})
\begin{equation}
4(k-1)\geq\left\Vert \{u,a_{i}\},V\smallsetminus R\right\Vert \geq4k-3-\left\Vert \{u,a_{i}\},R\right\Vert .\label{2.1}
\end{equation}
Thus $\left\Vert \{u,a_{i}\},R\right\Vert \geq1$. Hence $G[R]$ has
an edge, $|P|\geq2$, and $\left\Vert a_{i},P\right\Vert =\left\Vert a_{i},R\right\Vert =1$.\\
(\ref{clm:2.6}) By (\ref{2.1}) and (\ref{clm:2.5}),
%
 $\left\Vert \{u,a_{i}\},V\smallsetminus R\right\Vert \geq4|\mathcal{C}|-\left\Vert u,R\right\Vert $. Using (\ref{clm:2.2}), this implies the second assertion, and
$\|\{u,a_i\},C\| + 4(|\mathcal C|-1) \geq 4|\mathcal C| - \|u,R\|$
implies the first assertion. 
\end{proof}

\begin{figure}[ht]

\subcaptionbox{\label{subfig1.1}}
{\begin{tikzpicture}
\draw (0,0) circle (1cm);
\draw (0,-1) node[vertex, label=$y$,draw](y){};
\draw (-.707,-.707) node[vertex, label=$x$,draw](x){};
\draw (.707,-.707) node[vertex, label=$z$,draw](z){};
\draw (.7,-2) node[vertex, 
label=below:{$a_i$},draw](a){};
\draw (2.3,-2) node[vertex](b){};
\draw (a) -- (b);
\draw[dotted, line width=2pt] (a)--(b);
\draw (-.5,-2) node[vertex, label=below:{$u$}, draw](u){};
\draw[dotted, line width=2pt] (x) -- (u) -- (y);
\draw[dotted, line width=2pt]  (z) -- (a);
\draw[dotted, line width=2pt] (x) arc (225:270:1cm);
\end{tikzpicture}}
\hspace{1cm}
\subcaptionbox{\label{subfig1.2} }
{\begin{tikzpicture}
\draw (0,0) circle (1cm);
\draw (0,-1) node[vertex, label=$y$,draw](y){};
\draw (-.707,-.707) node[vertex, label=$z$,draw](z){};
\draw (.707,-.707) node[vertex, label=$x$,draw](x){};
\draw (1.3,-2) node[vertex, label=below:$a_i$,draw](a){};
\draw (-.5,-2) node[vertex, label=below:$a_{3-i}$, draw](b){};
\draw (a) -- (b);
\draw (-1.5,-2) node[vertex, label=below:$u$, draw](u){};
\draw (.75,-2) node[vertex](p){};
\draw[dotted, line width=2pt] (x) -- (a) -- (y);
\draw[dotted, line width=2pt]  (p) -- (b) -- (z) -- (u);
\draw[dotted, line width=2pt] (y) arc (270:315:1cm);
\end{tikzpicture}}

\caption{Claim \ref{clm:2}}
\end{figure}
 
\begin{claim}
\label{clm:4.6} $|P|\geq3$. In particular, $a_{1}a_{2}\not\in E(G)$. \end{claim}
\begin{proof}
Suppose $|P| \leq 2$. Then  $\|u,R\|\le 1$. As $|L|\geq3$, there is a
bud $c\in L\smallsetminus\{a_{1},a_{2}\}$. By Claim~\ref{clm:2}(\ref{clm:2.6}),
 there exists $C=z_{1}\dots z_{t}z_{1}\in\mathcal{C}$ such that  $\left\Vert \{c,a_{1}\},C\right\Vert =4$ and $\left\Vert \{c,a_{2}\},C \right\Vert \ge3$.

If $\left\Vert c,C\right\Vert =3$ then the edge between $a_1$ and $C$ 
 contradicts Claim~\ref{clm:2}(\ref{clm:2.1}). If $\left\Vert c,C\right\Vert =1$ then $\left\Vert \{a_1,a_2\},C\right\Vert =5$, contradicting Claim~\ref{clm:2}(\ref{clm:2.4}).
Therefore, we assume $\left\Vert c,C\right\Vert =2=\left\Vert a_{1},C\right\Vert $ and $\left\Vert a_{2},C\right\Vert\ge1$. By Claim~\ref{clm:2}(\ref{clm:2.1}),
$N(a_{1})\cup N(a_{2})= N(c)$, so
there exists $z_{i}\in N(a_{1})\cap N(a_{2})$ and $z_{j}\in N(c)-z_{i}$.
Then $a_{1}a_{2}z_{i}a_{1},cz_{j}z_{j\pm1}$ beats $C,P$
by \ref{o3}.\end{proof}
\begin{claim}
\label{at_most_3} Let  $c\in L-a_{1}-a_{2}$, $C\in\mathcal{C}$,
and $i\in[2]$. 
\begin{enumerate}[label=(\alph{enumi}), ref=\alph{enumi}]
\item \label{clm:3.1}$\left\Vert a_{1},C\right\Vert =3$ if and only if
$\left\Vert c,C\right\Vert =0$, and if and only if $\left\Vert a_{2},C\right\Vert =3$. 
\item \label{clm:3.2}There is at most one cycle $D\in\mathcal{C}$ with
$\left\Vert a_{i},D\right\Vert =3$. 
\item \label{clm:3.3}For every $C\in\mathcal{C}$, $\left\Vert a_{i},C\right\Vert \geq1$
and $\left\Vert c,C\right\Vert \leq2$.
\item \label{clm:3.4}If $\left\Vert \{a_{i},c\},C\right\Vert =4$ then
$\left\Vert a_{i},C\right\Vert =2=\left\Vert c,C\right\Vert .$
\end{enumerate}
\end{claim}
\begin{proof}
(\ref{clm:3.1}) If $\left\Vert c,C\right\Vert =0$ then by Claims~\ref{clm:1} and \ref{clm:2}(\ref{clm:2.6}),
$\left\Vert a_{i},C\right\Vert =3$. If $\left\Vert a_{i},C\right\Vert \geq3$
then by Claim~\ref{clm:2}(\ref{clm:2.2}), $\left\Vert c,C\right\Vert \leq1$. 
By Claim~\ref{clm:2}(\ref{clm:2.6}), $\left\Vert a_{3-i},C\right\Vert \geq2$, and 
by Claim~\ref{clm:2}(\ref{clm:2.4}), $\left\Vert c,C\right\Vert =0$.

(\ref{clm:3.2}) As  $c\in L$, $\left\Vert c,R\right\Vert \le 1$. Thus Claim~\ref{clm:2}(\ref{clm:2.6}) implies $\left\Vert c,D\right\Vert =0$
 for at most one cycle $D\in\mathcal{C}$.

(\ref{clm:3.3})
Suppose $\left\Vert c,C\right\Vert =3$. By Claim~\ref{clm:2}(\ref{clm:2.1}),
$\left\Vert \{a_{1},a_{2}\},C\right\Vert =0$. By Claims~\ref{clm:4.6}
and \ref{clm:2}(\ref{clm:2.4}): 
\[
4k-3\leq\left\Vert \{a_{1},a_{2}\},R\cup C\cup(V-R-C)\right\Vert \leq2+0+4(k-2)=4k-6,
\]
a contradiction. Thus $\left\Vert c,C\right\Vert \leq2$. Thus by Claim~\ref{clm:2}(\ref{clm:2.6}),
$\left\Vert a_{i},C\right\Vert \geq1$. 
 
 (\ref{clm:3.4})  Now (\ref{clm:3.4}) follows from (\ref{clm:3.1}) and \eqref{clm:3.3}.\end{proof}
\begin{claim}
\textcolor{red}{\label{Cl9/2}}$R$ has no isolated vertices.\end{claim}
\begin{proof}
Suppose $c\in L$ is isolated. Fix $C\in\mathcal{C}$.
By Claim~\ref{clm:2}(\ref{clm:2.6}), 
 $\left\Vert \{c,a_{1}\},C\right\Vert =4$.
By Claim~\ref{at_most_3}(\ref{clm:3.4}), $\left\Vert a_{1},C\right\Vert =2=\left\Vert c,C\right\Vert $; so $d(c)=2(k-1)$.
By Claim~\ref{clm:2}(\ref{clm:2.1}), $N(a_{1})\cap C=N(c)\cap C$.
Let $w\in V(C)\smallsetminus N(c)$. Then $d(w)\geq4k-3-d(c)=2k-1 = 2|\mathcal C|+1$.
Therefore, either $\|w,R\| \geq 1$ or $|N(w)\cap D|=3$ for some
$D\in\mathcal{C}$. In the first case, $c(C-w)c$ beats $C$
by \ref{o4}. 
In the second case, by Claim~\ref{at_most_3}(\ref{clm:3.3}) there exists some $x\in N(a_{1})\cap D$. Then $c(C-w)c,w(D-x)w$
 beats $C,D$ by \ref{o3}. \end{proof}
\begin{claim}
\label{clm:4.5} $L$ is an independent set.\end{claim}
\begin{proof}
Suppose $ $$c_{1}c_{2}\in E(L$). By Claim~\ref{clm:4.6}, $c_{1},c_{2}\notin P$.
By Claim~\ref{clm:2}(\ref{clm:2.6}) and using $k\geq3$, there is $C\in\mathcal{C}$
with $\left\Vert \{a_{1},c_{1}\},C\right\Vert =4$ and $\left\Vert \{a_{1},c_{2}\},C\right\Vert$, $ \left\Vert \{a_{2},c_{1}\},C\right\Vert\ge 3 $. 
By Claim~\ref{at_most_3}(\ref{clm:3.4}), $\left\Vert a_{1},C\right\Vert =2=\left\Vert c_{1},C\right\Vert$; so $\left\Vert a_{2},C\right\Vert$, $ \left\Vert c_{2},C\right\Vert \ge 1$.
By Claim~\ref{clm:2}(\ref{clm:2.1}), $N(a_1)\cap C,N(a_2)\cap C\subseteq N(c_1)\cap C$. Then there are distinct $x,y\in N(c_1)\cap C$ with $xa_1,xa_2,ya_1\in E$.
 If $xc_2\in E$ then $c_1c_2xc_1$, $ya_1Pa_2$ beats  $C,P$ by \ref{o3}. Else $a_1Pa_2xa_1$, $c_1(C-x)c_2c_1$ beats $C,P$ by \ref{o1}.
%
%
%
\end{proof}
\begin{claim}
\label{clm:3.5}If $|L|\geq3$ then for some $D\in\mathcal{C}$,  
$\left\Vert l,C\right\Vert =2$ for 
every $C\in\mathcal{C}-D$
and every $l\in L$.\end{claim}
\begin{proof}
Suppose some $D_{1},D_{2}\in\mathcal{C}$ and $l_{1},l_{2}\in L$
satisfy $D_{1}\ne D_{2}$ and $\left\Vert l_{1},D_{1}\right\Vert \ne2\ne\left\Vert l_{2},D_{2}\right\Vert $. 

\noindent CASE 1: $l_{j}\notin \{a_{1},a_{2}\}$ for some $j\in[2]$. Say
$j=1$. For $i\in[2]$: $\left\Vert \{a_{i},l_{1}\},D_{1}\right\Vert \ne4$
by Claim~\ref{at_most_3}(\ref{clm:3.4}); $\left\Vert \{a_{i},l_{1}\},D_{2}\right\Vert =4$
by Claim~\ref{clm:2}(\ref{clm:2.6}); $\left\Vert a_{i},D_{2}\right\Vert =2$
by Claim~\ref{at_most_3}(\ref{clm:3.4}). Then $l_{2}\notin\{a_{1},a_{2}\}$.
By Claim~\ref{clm:4.5}, $l_{1}l_{2}\not\in E(G)$. Claim~\ref{at_most_3}(\ref{clm:3.3})
yields the contradiction:
\[
4k-3\leq\left\Vert \{l_{1},l_{2}\},R\cup D_{1}\cup D_{2}\cup(V-R-D_{1}-D_{2})\right\Vert \leq2+3+3+4(k-3)=4k-4.
\]

\noindent CASE 2: $\{l_{1},l_{2}\}\subseteq\{a_{1},a_{2}\}$. Let
$c\in L-l_{1}-l_{2}$. As above, $\left\Vert \{l_{1},c\},D_{1}\right\Vert \ne4$,
and so $\left\Vert c,D_{2}\right\Vert =2=\left\Vert l_{1},D_{2}\right\Vert $. This implies $l_1 \neq l_2$.
By Claim~\ref{at_most_3}(\ref{clm:3.1},\ref{clm:3.3}), $\left\Vert l_{2},D_{2}\right\Vert =1$.
Thus $\left\Vert \{l_{2},c\},D_{1}\right\Vert =4$; so $\left\Vert c,D_{1}\right\Vert =2$, and
 $\left\Vert l_{1},D_{1}\right\Vert =1$. With 
Claim \ref{clm:4.6},
 this yields the contradiction:
\[
4k-3\leq\left\Vert \{l_{1},l_{2}\},R\cup D_{1}\cup D_{2}\cup(V-R-D_{1}-D_{2})\right\Vert \leq2+3+3+4(k-3)=4k-4.
\]
\end{proof}
\begin{claim}
\label{clm:5} $R$ is a subdivided star (possibly a path).\end{claim}
\begin{proof}
Suppose not. Then we claim $R$ has distinct leaves $c_{1},d_{1},c_{2},d_{2}\in L$
such that $c_{1}Rd_{1}$ and $c_{2}Rd_{2}$ are disjoint paths. Indeed, if
$R$ is disconnected then each component has two distinct leaves by
Claim~\ref{Cl9/2}. Else $R$ is a tree. As $R$ is not a subdivided
star, it has distinct vertices $s_{1}$ and $s_{2}$ with degree at
least three. Deleting the edges and interior vertices of $s_{1}Rs_{2}$
yields disjoint trees containing all leaves of $R$. Let $T_i$ be the tree containing $s_i$, and pick $c_i,d_{i}\in T_i$.

By Claim~\ref{clm:3.5}, using $k \geq 3$, there is a cycle $C\in\mathcal{C}$ such that 
$\left\Vert l,C\right\Vert =2$ for all $l\in L$. By Claim~\ref{clm:2}(\ref{clm:2.1}),
$N(a_{1})\cap C=N(l)\cap C=N(a_{2})\cap C=:\{w_{1},w_{3}\}$ for $l\in L-a_1 -a_2$.
Then replacing $C$ in $\mathcal C$ with $w_1c_1Rd_1w_1$ and $w_3c_2Rd_2w_3$
yields $k$ disjoint cycles.
\end{proof}
\begin{claim}
\label{clm:6} $R$ is a path or a star. \end{claim}
\begin{figure}[ht]

\subcaptionbox{\label{subfig3.1}}
{\begin{tikzpicture}[scale=0.8]
\draw(0,0) node[vertex,label=$r$,draw] (r){} ;
\draw(2.5,1) node[vertex,label=$p$,draw] (p){} ;
\draw(2.5,0) node[vertex,label=$a_2$,draw] (c){} ;
\draw(2.5,-1) node[vertex,label=$d$,draw] (d){} ;
\draw(3.2,1.5) node[vertex,label=$a_1$,draw] (l){} ;
\draw (r) -- (c) (d)--(r)--(p)--(l);
\end{tikzpicture}}
\subcaptionbox{\label{subfig3.2}}
{\begin{tikzpicture}[scale=0.8]
\draw(0,0) node[vertex,label=$r$,draw] (r){} ;
\draw(2.5,1) node[vertex,label={$p$},draw] (p){} ;
\draw(2.5,0) node[vertex,label={$a_2$},draw] (c){} ;
\draw(2.5,-1) node[vertex,label=$d$,draw] (d){} ;
\draw(3.2,1.5) node[vertex,label={$a_1$},draw] (l){} ;
\draw (r) -- (c) (d)--(r)--(p)--(l);
\draw(4,1) node[vertex,label=$w$,draw] (w1){} ;
\draw(4,0) node[vertex] (w2){} ;
\draw(4,-1) node[vertex] (w3){} ;
\draw (w1) -- (w2) -- (w3);
\draw (w3) to[out=45, in=-45] (w1);
\foreach \x in {1,2,3}
	{\draw (p) -- (w\x);
	\draw (l) -- (w\x);
	\draw (c) -- (w\x);}
\draw[dotted, line width=2pt] (p)--(w1)--(l)--(p);
\draw[line width=2pt, dotted] (c)--(w2)--(w3)--(c);
\end{tikzpicture}}
\quad
\subcaptionbox{\label{subfig3.3}}
{\begin{tikzpicture}[scale=0.8]
\draw(0,0) node[vertex,label=$r$,draw] (r){} ;
\draw(2.5,1) node[vertex,label={$p$},draw] (p){} ;
\draw(2.5,0) node[vertex,label={$l_1$},draw] (c){} ;
\draw(2.5,-1) node[vertex,label=$l_2$,draw] (d){} ;
\draw(3.2,1.5) node[vertex,label={$a_1$},draw] (l){} ;
\draw (d)--(r) -- (c) (r)--(p)--(l);
\draw(4,1) node[vertex,label=$w$,draw] (w1){} ;
\draw(4,0) node[vertex] (w2){} ;
\draw(4,-1) node[vertex] (w3){} ;
\draw (w1) -- (w2) -- (w3);
\draw (w3) to[out=45, in=-45] (w1);
\foreach \x in {1,2,3}
	{\draw (p) -- (w\x);}
\draw (c) -- (w1) -- (d);
\draw[dotted, line width=2pt] (c)--(w1)--(d)--(r) -- (c);
\draw[line width=2pt, dotted] (p)--(w2)--(w3)--(p);
\end{tikzpicture}}
\quad
\subcaptionbox{\label{subfig3.4}}
{\begin{tikzpicture}[scale=0.8]
\draw(0,0) node[shape=circle,inner sep=0,fill=black,minimum size=1.5mm,label=$r$,draw] (r){} ;
\draw(2.5,1) node[vertex,label={$p$},draw] (p){} ;
\draw(2.5,0) node[vertex,label={$l_1$},draw] (c){} ;
\draw(2.5,-1) node[vertex,label=$l_2$,draw] (d){} ;
\draw(3.2,1.5) node[vertex,label={$a_1$},draw] (l){} ;
\draw (d)--(r) -- (c) (r)--(p)--(l);
\draw(4,1) node[vertex,label=$w$,draw] (w1){} ;
\draw(4,0) node[vertex] (w2){} ;
\draw(4,-1) node[vertex] (w3){} ;
\draw (w1) -- (w2) -- (w3);
\draw (w3) to[out=45, in=-45] (w1);
\foreach \x in {1,2,3}
	{\draw (p) -- (w\x);}
\draw (c) -- (w1);
\draw (w2) -- (d) -- (w3);
\draw[dotted, line width=2pt] (c)--(w1)--(p)--(r) -- (c);
\draw[line width=2pt, dotted] (d)--(w2)--(w3)--(d);
\end{tikzpicture}}
\caption{Claim \ref{clm:6}}
\end{figure}

\begin{proof}
By Claim~\ref{clm:5}, $R$ is a subdivided star. If $R$ is neither
a path nor a star then there are vertices $r,p,d$ with $\left\Vert r,R\right\Vert \geq3$,
$\left\Vert p,R\right\Vert =2$, $d\in L-a_{1}-a_{2}$ and (say) $pa_{1}\in E$.
Then $a_{2}Rd$ is disjoint from $pa_{1}$ (see Figure \ref{subfig3.1}).
By Claim~\ref{at_most_3}(\ref{clm:3.3}), $d(d)\leq1+2(k-1)=2k-1$.
Then:
\begin{equation}\label{eqn:substar}
\left\Vert p,V-R\right\Vert \geq4k-3-\left\Vert p,R\right\Vert -d(d)\geq4k-5-(2k-1)=2k-4\geq2.
\end{equation}

In each of the following cases,  $R\cup C$ has two disjoint cycles, contradicting \ref{o1}.

\noindent CASE 1: $\left\Vert p,C\right\Vert =3$ for some $C\in\mathcal{C}$. Then $|C|=3$. 
By Claim~\ref{at_most_3}(\ref{clm:3.1}), if $\left\Vert d,C\right\Vert =0$
then $\left\Vert a_{1},C\right\Vert =3=\left\Vert a_{2},C\right\Vert $.
Then for $w\in C$, $wa_{1}pw$ and $a_{2}(C-w)a_{2}$ are disjoint
cycles (see Figure \ref{subfig3.2}). Else by Claim~\ref{at_most_3}(\ref{clm:3.3}),
$\|d,C\|$, $\|a_{2},C\|\in\{1,2\}$. By Claim~\ref{clm:2}(\ref{clm:2.6}),
$\|\{d,a_{2}\},C\|\geq3$, so there are $l_{1},l_{2}\in\{a_{2},d\}$
with $\left\Vert l_{1},C\right\Vert \ge1$ and $\left\Vert l_{2},C\right\Vert =2$;
say $w\in N(l_1)\cap C$. If $l_{2}w\in E$ then $wl_{1}Rl_{2}w$ and $p(C-w)p$
are disjoint cycles (see Figure \ref{subfig3.3}); else $l_{1}wpRl_{1}$
and $l_{2}(C-w)l_{2}$ are disjoint cycles (see Figure \ref{subfig3.4}).

\noindent CASE 2: There are distinct $C_{1},C_{2}\in\mathcal{C}$
with $\left\Vert p,C_{1}\right\Vert ,$ $\left\Vert p,C_{2}\right\Vert \geq1$.
By Claim~\ref{clm:3.5}, for some $i\in[2]$ and all $c\in L$, $\left\Vert c,C_{i}\right\Vert =2$.
Let $w\in N(p)\cap C_{i}$. If $wa_{1}\in E$ then $D:=wpa_{1}w$ is a cycle and
$G[(C_{i}-w)\cup a_2Rd]$ contains cycle disjoint from $D$. Else, if $w\in N(a_{2})\cup N(d)$,
say $w\in N(c)$, then $a_{1}(C_{i}-w)a_{1}$ and $cwpRc$ are disjoint
cycles. Else, by Claim~\ref{clm:1} there exist vertices $u\in N(a_{2})\cap N(d)\cap C_{i}$ and $v\in N(a_{1})\cap C_{i}-u$.
Then $ua_2Rdu$ and $a_{1}v(C_{i}-u)wpa_{1}$ are disjoint cycles.

\noindent CASE 3: Otherwise. Then using (\ref{eqn:substar}), $\left\Vert p,V-R\right\Vert =2=\left\Vert p,C\right\Vert $
for some $C\in\mathcal{C}$. In this case, $k=3$ and $d(p)=4$. By (H2),  $d(a_2),d(d)\ge 5$. Say $\mathcal C=\{C,D\}$. By Claim \ref{clm:2}(\ref{clm:2.2}), $\|\{a_2,d\},D\|\le4$. Thus,
\[
 \|\{a_{2},d\},C\|=\|\{a_2,d\},(V-R-D)\|\ge10-2-4=4.
\]
By Claim \ref{at_most_3}(\ref{clm:3.3}, \ref{clm:3.4}), $\left\Vert a_{2},C\right\Vert =\left\Vert d,C\right\Vert =2$ 
and  $\left\Vert a_{1},C\right\Vert \geq1$. Say $w\in N(a_{1})\cap C$. If
$wp\in E$ then 
$dRa_{2}(C-w)d$ contains a cycle disjoint from 
$wa_{1}pw$.
 Else, by Claim~\ref{clm:2}(\ref{clm:2.1}) there exists $x\in N(a_{2})\cap N(d)\cap C$. If $x\ne w$ then $xa_{2}Rdx$
and $wa_{1}p(C-x)w$ are disjoint cycles. Else $x=w$, and  $xa_2Rdx$ and $p(C-w)p$ are disjoint cycles. 
\end{proof}

\begin{lem}
\label{clm:7}$R$ is a path.\end{lem}
\begin{proof}
Suppose $R$ is not a path. Then it is a star with root $r$ and at
least three leaves, 
 any of which can play the role of
$a_{i}$ or a leaf in $L-a_{1}-a_{2}$. Thus Claim~\ref{at_most_3}(\ref{clm:3.3})
implies $\left\Vert l,C\right\Vert \in\{1,2\}$ for all $l\in L$ and $C\in\mathcal{C}$.
By Claim~\ref{clm:3.5} there is $D\in\mathcal{C}$ such that 
for all $l\in L$ and $C\in\mathcal{C}-D$,
$\left\Vert l,C\right\Vert =2$. By Claim~\ref{clm:2}(\ref{clm:2.6})
there is $l\in L$ such that for all $c\in L-l$, $\left\Vert c,D\right\Vert =2$. Fix distinct leaves $l',l''\in L-l$.

Let $Z=N(l')-R$ and $A=V\smallsetminus(Z\cup\{r\})$. By the first paragraph, every $C\in\mathcal{C}$
satisfies $|Z\cap C|=2$, so $|A|=|G|-2k+1$. For a contradiction,
we show that $A$ is independent. 

Note $A \cap R = L$, so by Claim~\ref{clm:4.5}, $A\cap R$ is independent. By Claim~\ref{clm:2}(\ref{clm:2.1}),
\begin{equation}\label{*}
\textrm{for all $c\in L$ and for all $C\in\mathcal{C}$, $N(c)\cap C\subseteq Z$.}
\end{equation}
Therefore, $\|L,A\|=0$. By Claim~\ref{clm:1}(c), for all $C\in\mathcal{C}$,
 $C\cap A$ is independent. Suppose, for a contradiction, $A$ is
not independent. Then there exist distinct $C_{1},C_{2}\in\mathcal{C}$,
$v_{1}\in A\cap C_{1}$, and $v_{2}\in A\cap C_{2}$ with $v_{1}v_{2}\in E$.
Subject to this choose  
$C_{2}$ 
with $\left\Vert v_{1},C_{2}\right\Vert $
maximum. Let $Z\cap C_{1}=\{x_{1},x_{2}\}$ and $Z\cap C_{2}=\{y_{1},y_{2}\}$. 

\noindent CASE 1: $\left\Vert v_{1},C_{2}\right\Vert \geq2$. Choose
$i\in[2]$ so that $\left\Vert v_{1},C_{2}-y_{i}\right\Vert \geq2$.
Then define $C_{1}^{*}:=v_{1}(C_{2}-y_i)v_{1}$, $C_{2}^{*}:=l'x_{1}(C_{1}-v_{1})x_{2}l'$,
and $P^{*}:=y_{i}l''rl$ (see Figure~\ref{fig:clm:7.1}). 
By \eqref{*}, $P^{*}$ is a path and $C_2^*$ is a cycle. Then $C_{1}^{*},C_{2}^{*},P^{*}$ beats $C_{1},C_{2},P$ by \ref{o3}.

\noindent CASE 2: $\|v_1,C_2\| \leq 1$. Then for all $C\in\mathcal{C}$, $\left\Vert v_{1},C\right\Vert \leq2$ 
 and $\left\Vert v_{1},C_{2}\right\Vert =1$;
so $\left\Vert v_{1},\mathcal{C}\right\Vert  = \|v_1,C_2\cup ( \mathcal C - C_2)\| \leq 1+2(k-2)= 2k-3$. By \eqref{*} $\left\Vert v_{1},L\right\Vert =0$
and $d(l)\leq2k-1$. By $(H2)$, $\|v_1,r\|=\|v_1,R\| = (4k-3)-\|v_1,\mathcal C\|-d(l) \leq 
(4k-3)-(2k-3)-(2k-1)=1$, and $v_{1}r\in E$.
 Let $C_{1}^{*}:=l'x_{1}(C_{1}-v_{1})x_{2}l'$,
$C_{2}^{*}:=l''y_{1}(C_{2}-v_{2})y_{2}l''$, and $P^{*}:=v_{2}v_{1}rl$ (see Figure~\ref{fig:clm:7.2}).
Then $C_{1}^{*},C_{2}^{*},P^{*}$ beats $C_{1},C_{2},P$
by \ref{o3}.
\end{proof}
\begin{figure}[ht]
\subcaptionbox{\label{fig:clm:7.1}}
{\begin{tikzpicture}
\draw (0,0) circle (1cm);
\draw (0,1) node[vertex, label=$v_1$,draw](v1){};
\draw (-.707,-.707) node[vertex, label=$x_1$,draw](x1){};
\draw (.707,-.707) node[vertex, label=$x_2$,draw](x2){};
\draw (3,0) circle (1cm);
\draw (3,1) node[vertex, label=$v_2$,draw](v2){};
\draw (3+-.707,-.707) node[vertex](y){};
\draw (3.707,-.707) node[vertex, label=$y_i$,draw](yi){};
\draw (0,-2) node[vertex, 
label=below:{$l'$},draw](l'){};
\draw (1.5,-2) node[vertex, 
label=below:{$r$},draw](r){};
\draw (3,-2) node[vertex, 
label=below:{$l''$},draw](l''){};
\draw (2.5,-2.5) node[vertex,
label=below:{$l$},draw](l){};
\draw (l')--(r)--(l'');
\draw (r)--(l);
\draw (x1)--(l')--(x2);
\draw (l'')--(yi);
\draw (v2)--(v1)--(y);
\draw[dotted, line width=2pt] (yi)--(l'')--(r)--(l);
\draw[dotted, line width=2pt] (x1)--(l')--(x2);
\draw[dotted, line width=2pt] (x1) arc (225:320:1cm);
\draw[dotted, line width=2pt] (v2) arc (90:225:1cm);
\draw[dotted, line width=2pt] (y)--(v1)--(v2);
\end{tikzpicture}
}
\hspace{2cm}
\subcaptionbox{\label{fig:clm:7.2} }
{\begin{tikzpicture}
\draw (0,0) circle (1cm);
\draw (0,1) node[vertex, label=$v_1$,draw](v1){};
\draw (-.707,-.707) node[vertex, label=$x_1$,draw](x1){};
\draw (.707,-.707) node[vertex, label=$x_2$,draw](x2){};

\draw (3,0) circle (1cm);
\draw (3,1) node[vertex, label=$v_2$,draw](v2){};
\draw (3+-.707,-.707) node[vertex,draw, label=$y_1$](y1){};
\draw (3.707,-.707) node[vertex, label=$y_2$,draw](y2){};

\draw (0,-2) node[vertex, 
label=below:{$l'$},draw](l'){};
\draw (1.5,-2) node[vertex, 
label=below:{$r$},draw](r){};
\draw (3,-2) node[vertex, 
label=below:{$l''$},draw](l''){};
\draw (2.5,-2.5) node[vertex, 
label=below:{$l$},draw](l){};
\draw (l')--(r)--(l'');
\draw (r)--(l);

\draw (x1)--(l')--(x2);
\draw (y1)--(l'')--(y2);
\draw (v2)--(v1) to[out=-45, in=100] (r);
\draw[dotted, line width=2pt] (v2)--(v1) to[out=-45, in=100] (r);
\draw[dotted, line width=2pt] (r)--(l);

\draw[dotted, line width=2pt] (x1)--(l')--(x2);
\draw[dotted, line width=2pt] (x1) arc (225:320:1cm);
\draw[dotted, line width=2pt] (y1) arc (225:320:1cm);
\draw[dotted, line width=2pt] (y1)--(l'')--(y2);
\end{tikzpicture}}
\caption{Claim \ref{clm:6}}
\end{figure}

\subsection{$\mathbf{|R|=3}$\label{three} } By Lemma~\ref{clm:7}, $R$ is a path, and by Claim~\ref{clm:4.6},  $|R|\geq3$. Next we
prove $|R|=3$. First, we prove a claim that will also be useful in later sections. 

\begin{claim}
\label{PC}Let $C$ be a cycle, $P=v_{1}v_{2}\dots v_{s}$ be a path in $R$,
and $1<i<s$. At most one of the following two statements holds.
\begin{itemize}
\item[(1)] (a) $\left\Vert x,v_{1}Pv_{i-1}\right\Vert \geq1$ for all $x\in C$
or (b) $\left\Vert x,v_{1}Pv_{i-1}\right\Vert \geq2$ for two $x\in C$;
\item[(2)]
(c) $\left\Vert y,v_{i}Pv_{s}\right\Vert \geq2$ for some $y\in C$
or (d) $N(v_{i})\cap C\ne\emptyset$ and $\left\Vert v_{i+1}Pv_{s},C\right\Vert \geq2$.
\end{itemize}
\end{claim}
\begin{proof}
Suppose (1) and (2) hold. If (c) holds then the disjoint graphs $G[v_{i}Pv_{s}+y]$
and $G[v_{1}Pv_{i-1}\cup C-y]$ contain cycles. Else $(d)$ holds,
but (c) fails; say $z\in N(v_{i})\cap C$ and $z\notin N(v_{i+1}Pv_{s})$.
If (a) holds then $G[v_{1}Pv_{i}+z]$ and $G[v_{i+1}Pv_{s}\cup C-z]$
contain cycles. If (b) holds then $G[v_{1}Pv_{i-1}+w]$ and $G[v_{i}Pv_{s}\cup C-w]$
contain cycles, where $\left\Vert w,v_{1}Pv_{i-1}\right\Vert \geq2$.
\end{proof}
Suppose, for a contradiction, $|R|\geq4$. Say $R=a_{1}a'_{1}a''_{1}\dots a''_{2}a'_{2}a_{2}$.
It is possible that $a''_{1}\in\{a''_{2},a'_{2}\}$, etc. Set $e_{i}:=a_{i}a'_{i}=\{a_{i},a'_{i}\}$
and $F:=e_{1}\cup e_{2}$. 
\begin{claim}
\label{7edges} If $C\in\mathcal{C}$, $h\in[2]$ and $\left\Vert e_{h},C\right\Vert \geq\left\Vert e_{3-h},C\right\Vert $
then $\left\Vert C,F\right\Vert \leq7$; if $\left\Vert C,F\right\Vert =7$
then 
\[
|C|=3,\,\left\Vert a_{h},C\right\Vert =2,~\left\Vert a'_{h},C\right\Vert =3,~\left\Vert a''_{h}Ra_{3-h},C\right\Vert =2,~\mathrm{and}~N(a_{h})\cap C=N(e_{3-h})\cap C.
\]
\end{claim}
\begin{proof}
We will repeatedly use Claim~\ref{PC} to obtain a contradiction
to \ref{o1} by showing that $G[C\cup R]$ contains two disjoint cycles.
Suppose $\left\Vert C,F\right\Vert \geq7$ and say $h=1$. Then $\|e_{1},C\|\geq4$.
There is $x\in e_{1}$ with $\left\Vert x,C\right\Vert \geq2$.
Thus $|C|\leq4$ by Claim~\ref{clm:1}, and if $|C|=4$ then no vertex in
$C$ has two adjacent neighbors in $F$. Then (1) holds with $v_1=a_1$ and $v_{i}=a'_{2}$,
even when $|C|=4$. 

If $\left\Vert e_{1},C\right\Vert =4$, as is the case when $|C|=4$,
then $\left\Vert e_{2},C\right\Vert \geq 3$. If $|C|=4$ there is a cycle $D:=yza_{2}'a_{2}y$
for some $y,z\in C$. As (a) holds, $G[a_{1}Ra''_{2}\cup C-y-z]$
contains another disjoint cycle. Thus, $|C|=3$. As (c) must fail with
$v_{i}=a'_{2}$, (a) and (c) hold for $v_{i}=a'_{1}$ and $v_{1}=a_{2}$,
a contradiction. Then $\left\Vert e_{1},C\right\Vert \geq5$. If $\left\Vert a_{1},C\right\Vert =3$
then $(a)$ and (c) hold with $v_1=a_1$ and $v_{i}=a'_{1}$. Now $\left\Vert a_{1},C\right\Vert =2$,
$\left\Vert a'_{1},C\right\Vert =3$ and $\left\Vert a''_{1}Ra_{2},C\right\Vert \geq 2$.
If there is $b \in P-e_1$ and $c\in N(b)\cap V(C)\smallsetminus N(a_{1})$
then $G[a'_{1}Ra_{2}+c]$ and $G[a_{1}(C-c)a_{1}]$ both contain cycles.
For every $b \in R-e_1$, $N(b) \cap C \subseteq N(a_1)$. Then if $\|a_1''Ra_2,C\| \geq 3$, (c) holds for $v_1=a_1$ and $v_1=a_1''$, contradicting that (1) holds.
Now $\|a_1''Ra_2,C\|=\|e_1,C\|=2$ and $N(a_{1})=N(e_{2})$.\end{proof}

\begin{lem}
\label{R=00003D3}$|R|=3$ and $m:=\max \{|C|:C\in \mathcal C\}=4$. \end{lem}
\begin{proof}

Let $t=|\{C\in\mathcal{C}:\|F,C\|\le6\}|$ and $r=|\{C\in\mathcal{C}:|C|\ge5\}|$.  It suffices to show $r=0$ and $|R|=3$: then $m\le4$, and $|V(\mathcal{C})| = |G|- |R|\ge3(k-1)+1$ implies some $C\in\mathcal C$ has length $4$. Choose $R$ so that:\\
 (P1) $R$ has as few low vertices as possible,
and subject to this, \\
(P2) $R$ has a low end if possible.

Let $C \in \mathcal C$. By Claim~\ref{7edges}, $\|F,C\|\le 7$. By  Claim \ref{clm:1}, if $|C|\ge5$ then $\|a,C\|\le1$ for all $a\in F$; so
 $\|F,C\|\le4$. Thus $r\le t$. Hence
\begin{equation}
2(4k-3)\le \left\Vert F,(V\smallsetminus R)\cup R\right\Vert
\le7(k-1)-t-2r+6\le7k-t-2r-1. 
\label{eq:7edges}
\end{equation}
Therefore, $5-k\ge t+2r\ge3r\ge0$. 
Since $k \geq 3$, this yields $3r \leq t+2r \leq 2$, so
 $r=0$ and $t\le 2$, with $t=2$ only if $k=3$. 

CASE 1: $k-t \geq 3$.  That is, there exist distinct cycles $C_1,C_2\in\mathcal{C}$ with $\|F,C_i\|\ge7$. 
In this case, $t \le 1$:
 if $k=3$ then $\mathcal C=\{C_1,C_2\}$ and $t=0$; if $k>3$ then $t<2$.
For  both $i\in[2]$, Claim~\ref{7edges} yields $\left\Vert F,C_i\right\Vert =7$,  $|C_i|=3$, and
there is $x_i\in V(C_i)$ with $\left\Vert x_i,R\right\Vert =1$
and $\left\Vert y,R\right\Vert =3$ for both $y\in V(C_i-x_i)$.
Moreover, there is a  unique index $j=\beta(i)\in[2]$ with $\left\Vert a'_{j},C_{i}\right\Vert =3$.
For $j \in [2]$, put $I_{j}:=\{i\in[2]:\beta(i)=j\}$; that is, $I_j=\{i \in [2] : \|a_j',C_i\|=3\}$.
 Then $V(C_{i})-x_{i}= N(a_{\beta(i)})\cap C_i=N(e_{3-\beta(i)})\cap C_i$.
 As $x_{i}a_{\beta(i)}\notin E$,
one of $x_{i},a_{\beta(i)}$ is high. As we can switch $x_{i}$ and
$a_{\beta(i)}$ (by replacing $C_i$ with $a_{\beta(i)}(C_i-x_i)a_{\beta(i)}$ and $R$ with $R-a_{\beta(i)}+x_i$), we may assume $a_{\beta(i)}$ is high. 

Suppose $I_{j}\ne\emptyset$ for both $j\in[2]$; 
say $\|a_1',C_1\|=\|a_2',C_2\|=3$.
Then  for all $B\in \mathcal C$ and $j\in[2]$, $a_j$ is high, and either $\|a_j,B\|\le2$ or $\|F,B\|\leq6$. Since $t \leq 1$, we get
\[2k-1\le d(a_j)=\|a_j,B\cup F\|+\|a_j,\mathcal C-B\|\le\|a_j,B\|+1+2(k-2)+t\le 2k-2+\|a_j,B\|.\]
Thus $N(a_j)\cap B\ne\emptyset$ for all $B\in \mathcal C$. Let $y_{j}\in N(a_{3-j})\cap C_j$.
 Then using Claim~\ref{7edges}, $y_j \in N(a_{j})$, and
  $a'_{1}(C_{1}-y_{1})a'_{1},$ $a'_{2}(C_{2}-y_{2})a'_{2},$ $a_{1}y_{1}a_{2}y_{2}a_{1}$ beats $C_1,C_2$ by \ref{o1}.

Otherwise, say $I_{1}=\emptyset$.   If $B\in\mathcal C$ with $\|F,B\|\le6$ then
 $\|e_1,B\|+2\|a_2,B \| \le \|F,B\|+\|a_2,B\| \le9$. Thus, using Claim~\ref{7edges},
\begin{align*}
2(4k-3)&\le d(a_1) +d(a'_1)+2d(a_2) = 
5+\|e_1,\mathcal C\|+2\|a_2,\mathcal C\|
 \le5+ 6(k-1-t)+9t
 \\\Rightarrow  2k&\leq 
 5+3t.
\end{align*}
Since $k-t \geq 3$ (by the case), we see 
$3(k-t)+(5+3t) \geq 3(3)+2k$ and so $k \geq 4$.
Since $t \leq 1$, in fact $k=4$ and $t=1$, and equality holds throughout: say  $B$ is the unique cycle in $\mathcal C$ with $\|F,B\|\le6$. 
 Then $\|a_2,B\|=\|e_1,B\|=3$. 
 Using Claim~\ref{7edges}, $d(a_1)+d(a_1') = \|e_1,R\|+\|e_1,\mathcal C - B\|+\|e_1,B\| = 3+4+3=10$, and
 $d(a_1),d(a_2) \geq (4k-3) - d(a_2) = 13-(1+4+3)=5$, so
 $d(a_1)=d(a_2)=5$. Note $a_1$ and $a_2$ share no neighbors: they share none in $R$ because $R$ is a path, they share none in $\mathcal C - B$ by Claim~\ref{7edges}, and they share no neighbor $b \in B$ lest $a_1a_1'ba_1$ and $a_2(B-b)a_2$ beat $B$ by \ref{o1}. Thus every vertex in $V-e_1$ is high.
 
Since $\|e_1,B\| =3$, first suppose $\|a_1,B\| \geq2$, say $B-b \subseteq N(a_1)$. Then 
$a_1(B-b)a_1$, $a_1'a_2'a_1b$ beat $B,R$ by (P1) (see Figure~\ref{2.14subfig1.1}). Now suppose $\|a_1',B\|\geq 2$, this time with 
 $B-b \subseteq N(a_1')$. Since $d(a_1)=5$ and $\|a_1,R \cup B\| \leq 2$, there exists $c \in C \in \mathcal C -B$ with $a_1c \in E(G)$. Now $c \in N(a_2)$ by Claim~\ref{7edges}, so
 $a_1'(B-b)a_1'$, $a_2'(C-c)a_2'$, and $a_1ca_2b$ beat $B,C$, and $R$ by (P1)  (see Figure~\ref{2.14subfig1.2}). 
 
\begin{figure}[ht]
\subcaptionbox{\label{2.14subfig1.1}}
{\begin{tikzpicture}
\foreach \x in {0,...,3}
{\draw (1.5*\x,0) node[vertex](p\x){};}
\draw (p0)--(p1)--(p2)--(p3);
\draw (p0) node[label=below left:$a_1$] {};
\draw (p1) node[label=below:$a_1'$] {};
\draw (p2) node[label=below:$a_2'$] {};
\draw (p3) node[label=below right:$a_2$] {};
\foreach \x in {0,1,2}
{\draw (2.25,-2) + (-30+120*\x:.6cm) node[vertex](b\x){};
\draw[ultra thin] (p3)--(b\x);}
\draw (b0)--(b1)--(b2)--(b0);
\draw (b0) node[label=right:$b$]{};
\draw (b2)--(p0)--(b1);
\draw[dashed, line width=2pt] (b2)--(p0)--(b1)--(b2);
\draw[dashed, line width=2pt] (p1)--(p2)--(p3)--(b0);
\end{tikzpicture}}
\hskip 2cm
\subcaptionbox{\label{2.14subfig1.2} }
{\begin{tikzpicture}
\foreach \x in {0,...,3}
{\draw (1.5*\x,0) node[vertex](p\x){};}
\draw (p0)--(p1)--(p2)--(p3);
\draw (p0) node[label=below:$a_1$] {};
\draw (p1) node[label=below left:$a_1'$] {};
\draw (p2) node[label=below:$a_2'$] {};
\draw (p3) node[label=below right:$a_2$] {};
\foreach \x in {0,1,2}
{\draw (2.25,-2) + (-30+120*\x:.6cm) node[vertex](b\x){};
\draw[ultra thin] (p3)--(b\x);}
\draw (b0)--(b1)--(b2)--(b0);
\draw (b0) node[label=right:$b$]{};
\foreach \x in {0,1,2}
{\draw (3,1.5) + (-30+120*\x:.6cm) node[vertex](d\x){};}
\draw (d0)--(d1)--(d2)--(d0);
\foreach \x in {0,1,2}{\draw[ultra thin] (d\x)--(p2);}
\foreach \x in {0,2}{\draw[ultra thin] (d\x)--(p3);}
\draw (b2)--(p1)--(b1);
\draw (p0)--(d2);
\draw[dashed, line width=2pt] (b2)--(p1)--(b1)--(b2);
\draw[dashed, line width=2pt] (p0)--(d2)--(p3)--(b0);
\draw (d2) node[label=above left:$c$]{};
\draw[dashed, line width=2pt] (p2)--(d1)--(d0)--(p2);
\end{tikzpicture}}
\caption{Lemma \ref{R=00003D3}, Case 1}
\end{figure}

CASE 2:  $k-t\leq 2$. That is, $\|F,C\|\leq 6$ for all but at most one $C \in \mathcal C$. Then, since $5-k\geq t$, we get $k=3$ and $\left\Vert F,V\right\Vert \leq19$. Say $\mathcal{C}=\{C,D\}$, so
$\|F,C\cup D\|\ge2(4k-3) -\|F,R\|=2(4\cdot3-3)-6=12$. By Claim~\ref{7edges}, 
$\|F,C\|$, $\|F,D\|\le7$. Then $\|F,C\|$, $\|F,D\|\ge5$. If $|R|\geq 5$, then for the (at most two) low vertices in $R$, we can
choose distinct vertices in $R$ not adjacent to them. Then
$\|R,V-R\| \geq  5|R|-2-\|R,R\| =3|R|$. Thus we may assume $\|R,C\|\geq
 \lceil 3|R|/2\rceil \geq |R|+3\geq 8$. Let $w'\in C$ be such that  $q=\|w',R\|=\max\{\|w,R\|\,:\,w\in C\}$.
Let $N(w')\cap R=\{v_{i_1},\ldots,v_{i_q}\}$ with $i_1<\ldots<i_q$. Suppose $q\geq 4$. If $\|v_1Rv_{i_2}, C-w'\|\geq 2$
or $\|v_{i_2+1}Rv_{s}, C-w'\|\geq 2$, then $G[C\cup R]$ has two disjoint cycles. Otherwise, $\|R,C-w'\|\leq 2$,
contradicting $\|R,C\|\geq |R|+3$. Similarly, if $q=3$, then
$\|v_1Rv_{i_2-1}, C-w'\|\leq 1$ and  $\|v_{i_2+1}Rv_{s}, C-w'\|\leq 1$ yielding
$\|v_{i_2},C\| = \|R,C\|-\|(R-v_{i_2}),C-w'\|-\|R-v_{i_2},w'\|\geq (|R|+3)-2-(3-1)\geq 4$, a contradiction to Claim~\ref{clm:1}(a). Therefore, $q\leq 2$, and hence
$|R|+3\leq\|R,C\|\leq 2|C|$. It follows that $|R|=5$, $|C|=4$ and $\|w,R\|=2$ for each $w\in C$.
This in turn yields that $G[C\cup R]$ has no triangles and $\|v_i,C\|\leq 2$ for each $i\in [5]$.
By  Claim~\ref{7edges}, $\|F,C\|\leq 6$, so $\|v_3,C\|=2$. Thus we may assume that for some $w\in C$,
$N(w)\cap R=\{v_1,v_3\}$. Then $\|e_2,C\|=\|e_2,C-w\|\leq 1$, lest there exist a cycle disjoint from $wv_1v_2v_3w$ in $G[C \cup R]$. Therefore, $\|e_1,C\|\geq 8-1-2=5$, 
a contradiction to Claim~\ref{clm:1}(b). This yields $|R|\leq 4$.


\begin{claim}
\label{L-cl}Either $a_{1}$ or $a_{2}$ is low. \end{claim}
\begin{proof}
Suppose $a_{1}$ and $a_{2}$ are high. Then since $\left\Vert R,V\right\Vert \leq19$, we may assume
 $a'_{1}$ is low. Suppose
there is $c\in C$ with $ca_{2}\in E$ and $\left\Vert a_{1},C-c\right\Vert \geq 2$. If $a'_{1}c\in E$, then
$R\cup C$ contains two disjoint cycles; 
so $a'_{1}c\notin E$ and hence $c$ is high. Thus either $a_{1}(C-c)a_{1}$ is
shorter than $C$ or  the pair $a_{1}(C-c)a_{1}$,
$ca_{2}a'_{2}a'_{1}$ beats $C,R$ by $(P2)$. Thus if $ca_{2}\in E$ then $\left\Vert a_{1},C-c\right\Vert \leq 1$. 
As $a_{2}$
is high, $\left\Vert a_{2},C\right\Vert \geq1$ and hence $\left\Vert a_{1},C\right\Vert  = \|a_1,C\smallsetminus N(a_2)\|+\|a_1,N(a_2)\|\leq2$. Similarly, $\left\Vert a_{1},D\right\Vert \leq2$.
Since $a_1$ is high, we see 
 $\left\Vert a_{1},C\right\Vert =\left\Vert a_{1},D\right\Vert =2$,
and $d(a_{1})=5$. Hence 
\begin{equation}\label{0313}
N(a_{2})\cap C\subseteq N(a_{1})\cap C\quad\mbox{and}\quad N(a_{2})\cap  D\subseteq N(a_{1})\cap  D.
\end{equation}
As $a_2$ is high, $d(a_{2})=5$ and in~\eqref{0313} equalities hold. Also $d(a_{1}')=4\leq d(a'_{2})$.

If there are $c\in C$ and $i\in[2]$ with $ca_{i},ca'_{i}\in E$
then by \ref{o2}, $|C|=3$. Also 
$ca'_{i}a{}_{i}c$, $a_{3-i}'a_{3-i}(C-c)$ beats $C,R$
by either (P1) or  (P2). (Recall $N(a_1) \cap C = N(a_2) \cap C$ and neighbors of $a_2$ in $C$ are high.) Then $N(a_{i})\cap N(a_{i}')=\emptyset$. Thus the set
$N(a_1)-R=N(a_2)-R$ contains no low vertices.
Also, if $\left\Vert a'_{1},C\right\Vert \geq1$
then $|C|=3$: else $C$ has the form $c_{1}c_{2}c_{3}c_{4}c_{1}$,
where $a_{1}c_{1},a_{1}c_{3}\in E$, and so $a_{1}a_{1}'c_{1}c_{2}a_{1}$,
$c_{3}c_{4}a_{2}a'_{2}$ beats $C,R$ by either (P1) or (P2). 
Thus $|C|=3$ and $a'_{1}c\in E$ for some $c\in V(C)-N(a_{1})$. 
If $\|a_2',C\| \geq 1$, we have disjoint cycles
$ca'_{1}a'_{2}c$,
$a_{1}(C-c)a_{1}$ and $D$. Then $\|a_1',C\|=0$, so $d(a_1') \leq2+|D \smallsetminus N(a_1)| \leq 4$. Now $a_1'$ and $a_2'$ are symmetric, and we have proved that $\left\Vert a'_{1},C\right\Vert+ \|a_2',C\| \leq 1$. Similarly, $\left\Vert a'_{1},D\right\Vert + \|a_2',D\| \leq 1$, a contradiction to $d(a'_1),d(a'_2)\geq 4$.
\end{proof}

By Claim~\ref{L-cl}, we can choose notation so that $a_{1}$ is low.
\begin{claim}
\noindent \label{C_high} If $a'_{1}$
is low then each  $v\in V\smallsetminus e_{1}$ is high. \end{claim}
\begin{proof}
\noindent Suppose $v \in V-e_1$ is low. Since $a_1$ is low, all vertices in $R-e_1$ are high, so $v\in C$ for some $C \in \mathcal C$.
Then $C':=ve_{1}v$ is a cycle and so by \ref{o2}, $|C|=3$. 
Since $a_{2}$ is high, $\left\Vert a_{2},C\right\Vert \geq1$. As $v$ is low, $va_{2}\notin E$.
Since
$a'_{1}$ is low, it is adjacent to the low vertex $v$, and $\left\Vert a'_{1},C-v\right\Vert \leq1$.
Then $C',a_{2}'a_{2}(C-v)$ beats $C,R$ by (P1).\end{proof}

\begin{claim}
\label{RC3cl} If $|C|=3$ and $\left\Vert e_{1},C\right\Vert $, $\left\Vert e_{2},C\right\Vert \geq3$,
then either\\
 (a) $\left\Vert c,e_{1}\right\Vert =1=\left\Vert c,e_{2}\right\Vert $
for all $c\in V(C)$ or\\
 (b) $a_{1}'$ is high and there is $c\in V(C)$
with $\left\Vert c,R\right\Vert =4$ and $C-c$ has a low vertex. \end{claim}
\begin{proof}
If $(a)$ fails then $\left\Vert c,e_{i}\right\Vert =2$ for some
$i\in[2]$ and $c\in C$. 
 If $\left\Vert e_{3-i},C-c\right\Vert \geq2$
then there is a cycle $C'\subseteq C\cup e_{3-i}-c$, and $R\cup C$ contains disjoint cycles $ce_{i}c$ and $C'$.
Else, 
$$\left\Vert c,R\right\Vert =\left\Vert c,e_i\right\Vert+(\left\Vert C,e_{3-i}\right\Vert-\left\Vert C-c,e_{3-i}\right\Vert) \geq 2+(3-1)=4=|R|.$$
If $C-c$ has no low vertices then $ce_{1}c$, $e_{2}(C-c)$ beats $C,R$ by (P1). Then $C-c$ contains a low vertex $c'$. If $a_{1}'$ is low then $c'a_{1}'a_{1}c'$ and $ca_{2}a_{2}'c$ are disjoint cycles.
Thus, (b) holds. 
\end{proof}
CASE 2.1: $|D|=4$. 
By \ref{o2},
$G[R\cup D]$ does not contain a $3$-cycle.  Then $5\leq d(a_{2})\leq3+\left\Vert a_{2},C\right\Vert \leq6$.
Thus $d(a_{1}),d(a_{1}')\geq3$.

Suppose $\left\Vert e_{1},D\right\Vert \geq3$.  Pick
$v\in N(a_{1})\cap D$ with minimum degree, and $v'\in N(a_{1}')\cap D$.
Since  $N(a_{1})\cap D$
and $N(a'_{1})\cap D$ are nonempty, disjoint and independent, we see $vv'\in E$.
Say $D=vv'ww'v$. As $D=K_{2,2}$ and low vertices are adjacent, $D':=a_{1}a_{1}'v'va_{1}$
is a $4$-cycle and $v$ is the only possible low vertex in $D$.
Note $a_{1}w\notin E$: else $a_{1}ww'va_{1}$, $v'a'_{1}a'_{2}a_{2}$ 
beats $D,R$ by (P1). As $\left\Vert e_{1},D\right\Vert \geq3$,
$a'_{1}w'\in E$. Also note $\left\Vert e_{2},ww'\right\Vert =0$:
else $G[a_{2},a_{2}',w,w']$ contains a $4$-path $R'$, and $D',R'$
 beats $D,R$ by (P1). Similarly, replacing
$D'$ by $D'':=a_{1}a_{1}'w'va_{1}$ yields $\|e_2,v'\|=0$. Then
$\left\Vert e_1\cup e_2,D\right\Vert \leq3+1=4$, a contradiction.
Thus
\begin{equation}\label{03132}
\left\Vert e_{1},D\right\Vert \leq2\quad\mbox{and so}\quad \left\Vert R,D\right\Vert\leq 6.
\end{equation}

Suppose $d(a'_{1})=3$. Then $\left\Vert a'_{1},D\right\Vert \leq1$.
Then there is $uv\in E(D)$ with $\left\Vert a'_{1},uv\right\Vert =0$.
Thus $d(u),d(v),d(a_{2})\geq6$, and $\left\Vert a_{2},C\right\Vert =3$.
Now $|C|=3$, $|G|=11$, and there is $w\in N(u)\cap N(v)$. If $w\in C$
put $C'=a_{2}(C-w)a_{2}$; else $C'=C$. In both cases, $|C'|=|C|$ and $|wuvw|=3<|D|$,
so $C'$, $wuvw$ beats $C,D$ by \ref{o2}.
Thus $d(a'_{1})\geq 4$.
 If $d(a_{1})=3$ then $d(a_{2}),d(a_{2}')\geq 9-3=6$, 
and $\|a_2 ,C\|\ge3$. By~\eqref{03132},
$$\left\Vert R,C\right\Vert \geq 3+4+6+6-\left\Vert R,R\right\Vert- \left\Vert R,D\right\Vert\geq 19-6-6=7,$$ 
contradicting Claim~\ref{7edges}.
Then $d(a_{1})=4\leq d(a'_{1})$ and by~\eqref{03132}, $\left\Vert e_{1},C\right\Vert \geq3$.
Thus~\eqref{03132} fails for $C$ in place of $D$; so $|C|=3$.
As $\left\Vert a_{2},C\right\Vert \geq2$ and $\left\Vert a'_{2},C\right\Vert \geq1$, Claim~\ref{RC3cl} implies 
either (a) or (b) of Claim~\ref{RC3cl} holds. If (a) holds then (a) and (d) of Claim~\ref{PC} both hold, and so $G[C\cup R]$ has two disjoint cycles. Else, Claim~\ref{RC3cl} gives $a'_1$ is high and there is $c\in\mathcal C$ with $\|c,R\|=4$. As $a'_1$ is high, $\|R,C\|\ge7$. Now  $\|c,R\|=4$ contradicts Lemma~\ref{7edges}.
 
CASE 2.2: $|C|=|D|=3$  and $\left\Vert R,V\right\Vert =18$.
Then $d(a_{1})+d(a'_{2})=9=d(a_{1}')+d(a_{2})$, $a_{1}$ and $a'_{1}$
are low, and by Claim~\ref{C_high} all other vertices are high.
$ $Moreover, $d(a'_{1})\leq d(a{}_{1})$, since 
\[
18=\left\Vert R,V\right\Vert =d(a_{1}')-d(a_{1})+2d(a_{1})+d(a_{2}')+d(a{}_{2})\geq d(a_{1}')-d(a_{1})+9+9.
\]

Suppose $d(a_{1}')=2$. Then $d(v)\geq 7$ for all $v\in V-a_{1}a_{1}'a'_{2}$. In particular, $C \cup D \subseteq N(a_2)$.
If $d(a_{1})=2$ then $d(a'_{2})\geq 7$, and $G=\mathbf{Y_{1}}$. Else
$\left\Vert a_{1},C\cup D\right\Vert \geq2$. If there is $c\in {C}$
with $V(C)-c\subseteq N(a_{1})$, then $a_{1}(C-c)a_{1}$, $a'_{1}a'_{2}a_{2}c$
 beats $C,R$ by (P1). Else $d(a_{1})=3,$ $d(a'_{2})=6$,
and there are $c\in C$ and $d\in D$
with $c,d\in N(a_{1})$. If $ca'_{2}\in E$ then $C\cup R$ contains
disjoint cycles $a_{1}ca_{2}'a_{1}'a_{1}$ and
$a_{2}(C-c)a_{2}$, so assume not.
Similarly, assume $da'_{2}\notin E$. Since $d(d) \geq 7$ and $a_1',a_2' \not\in N(d)$, we see $cd \in E(G)$. Then there are three disjoint
cycles $a'_{2}(C-c)a'_{2}$, $a_{2}(D-d)a_{2}$, and $a_{1}cda_{1}$.$ $
Thus $d(a'_{1})\geq3$.

Suppose $d(a'_{1})=3$. Say  
$a'_{1}v\in E$ for some $v\in D$. As $d(a_{2})\geq6$, $\left\Vert a_{2},D\right\Vert \geq2$.
Then $e_{2}+D-v$ contains a $4$-path $R'$. Thus $a_{1}v\notin E$:
else $ve_{1}v,R'$  beats $D,R$ by (P1). Also $\left\Vert a_{1},D-v\right\Vert \leq1$:
else $a_{1}(D-v)a_{1},va'_{1}a'_{2}a_{2}$ beats $D,R$
by (P1). Then $\left\Vert a_{1},D\right\Vert \leq1$.

Suppose $\left\Vert a_{1},C\right\Vert \geq2$.
  Pick $c\in C$ with $C-c\subseteq N(a_{1})$. Then
  \begin{equation}\label{star_p_12}a_{2}c\notin E:\end{equation}
else $a_{1}(C-c)a_{1}$, $a'_{1}a'_{2}a_{2}c$ beats $C,R$
by (P1). Then $\left\Vert a_{2},C\right\Vert =2$ and $\left\Vert a_{2},D\right\Vert =3$.
Also $a_{1}c\notin E$: else picking a different $c$ violates \eqref{star_p_12}.
As $a'_1c\notin E$, $\left\Vert c,D\right\Vert =3$ and $a_2'c \in E(G)$. Then $a_1(C-c)a_1$, $a_{2}(D-v)a_{2}$ and $cva'_{1}a'_{2}c$ are disjoint cycles.
 Otherwise, $\left\Vert a_{1},C\right\Vert \leq1$
and $d(a_{1})\leq3$.  Then $d(a_1)=3$ since $d(a_1) \geq d(a_1')$.

Now $d(a'_{2})=6$. Say $D=vbb'v$ and
$a_{1}b\in E$. As $b'a'_1\notin E$, $d(b')\ge9-3=6$.
Since  $\|e_2,V\|=12$, we see that $a_2$ and $a'_2$ have three common neighbors. If one is $b'$ then $D':=a_1a'_1vba_1$, $b'e_2b'$, and $C$ are disjoint cycles; else $\|b',C\|=3$ and there is $c'\in C$ with $\|c',e_2\|=2$. Then $D'$, $c'e_2c'$ and $b'(C-c')b'$ are disjoint cycles.
 Thus, $d(a'_{1})=4$. 

Since $a_1$ is low and $d(a_1) \geq d(a_1')$, we see $d(a_1)=d(a_1')= 4$ and $\|\{a_1,a_1'\},C \cup D\| = 5$, so we may assume $\left\Vert e_{1},C\right\Vert \geq3$.
If  $\left\Vert e_{2},C\right\Vert \geq3$, then because $a_1'$ is low,
Claim~\ref{RC3cl}(a) holds. Now, $V(C)\subseteq N(e_1)$ and there is $x\in e_{1}=xy$ with
$\left\Vert x,C\right\Vert \geq2$. First suppose $\left\Vert x,C\right\Vert =3$.
As $x$ is low, $x=a_{1}$. Pick $c\in N(a_{2})\cap C$, which exists because $\|a_2,C \cup D\|\geq 4$. Then $a_{1}(C-c)a_{1}$,
$a'_{1}a'_{2}a_{2}c$ beats $C,R$ by (P1). Now suppose
$\left\Vert x,C\right\Vert =2$. Let $c\in C\smallsetminus N(x)$.
Then $x(C-c)x$, $yce_{2}$ beats $C,R$ by (P1).

CASE 2.3: $|C|=|D|=3$ and $\left\Vert R,V\right\Vert =19$.
Say  $\left\Vert C,R\right\Vert =7$
and $\left\Vert D,R\right\Vert =6$. 

CASE 2.3.1: $a'_{1}$ is low.  
Then $\|a_1',C \cup D\|\leq4-\|a_1',R\|=2$, so by Claim~\ref{7edges}, $\|e_2,C\|=5$ with $\|a_2,C\|=2$.
 Then $5\leq d(a_{2})\leq6$.

If $d(a_{2})=5$ then $d(a_{1})=d(a'_{1})=4$ and $d(a'_{2})=6$.
Then $\left\Vert a_{2},D\right\Vert =2$ and $\left\Vert a'_{2},D\right\Vert =1$.
Say $D=b_{1}b_{2}b_{3}b_{1}$, where $a_{2}b_{2},a_{2}b_{3}\in E$.
As $a'_{1}$ is low, (a) of Claim~\ref{RC3cl} holds. Then $\left\Vert b_{1},a_{1}a'_{1}a'_{2}\right\Vert =2$,
and there is a cycle $D'\subseteq G[b_{1}a_{1}a'_{1}a'_{2}]$.
Then $a_{2}(D-b_{1})a_2$ and $D'$ are disjoint.

If $d(a_{2})=6$ then $\left\Vert a_{2},D\right\Vert =3$. Let $c_1 \in C-N(a_2).$ By Claim~\ref{7edges}, $\left\Vert c_{1},R\right\Vert =1$, so
$c_{1}$ is high, and $\left\Vert c_{1},D\right\Vert \geq2$.
If $\left\Vert a_{2}',D\right\Vert \geq1$, then (a) and (d) hold in Claim~\ref{PC} for $v_1=a_2$ and $v_i = a_2'$, so $G[D\cup c_{1}a'_{2}a_{2}]$
has two disjoint cycles, and $c_{2}e_{1}c_{3}c_{2}$ contains a third.
Therefore, assume $\|a_2',D\|=0$, and so $d(a'_{2})=5$. 
Thus $d(a_{1})=d(a'_{1})=4$.
 Again, $\left\Vert e_{1},D\right\Vert =3=\left\Vert a_{2},D\right\Vert $.
Now there are $x\in e_{1}$ and $b\in V(D)$ with $D-b\subseteq N(x)$.
As $a'_{1}$ is low and has two neighbors in $R$, if $\left\Vert x,D\right\Vert =3$
then $x=a_{1}$. Anyway, using Claim~\ref{RC3cl}, $G[R+b-x]$ contains a $4$-path $R'$, and
$x(D-b)x$, $R'$ beats $D,R$ by (P1). 

CASE 2.3.2: $a'_{1}$ is high. Since $19=\left\Vert R,V\right\Vert\geq d(a_1)+d(a_1')+2(9-d(a_1))\geq 23-d(a_1)$,
we get
 $d(a_{1})=4$ and $d(a'_{1})=d(a'_{2})=d(a_{2})=5$.
  Choose
notation so that $C=c_{1}c_{2}c_{3}c_{1}$, $D=b_{1}b_{2}b_{3}b_{1}$,
and $\left\Vert c_{1},R\right\Vert =1$. By Claim~\ref{7edges},  there
is $i\in[2]$ with $\left\Vert a_{i},C\right\Vert =2$, $\left\Vert a'_{i},C\right\Vert =3$,
and $a_{i}c_{1}\notin E$. If $i=1$ then every low vertex is in $N(a_{1})-a_{1}'\subseteq D\cup C'$,
where $C'=a_{1}c_{2}c_{3}a_{1}$. Then $C'$, $c_{1}a'_{1}a'_{2}a_{2}$
 beats $C,R$ by (P1). Thus let $i=2$. 
Now $\left\Vert a_{2},C\right\Vert =2=\left\Vert a_{2},D\right\Vert $.

Say $a_{2}b_{2},a_{2}b_{3}\in E$. Also $\left\Vert a'_{2},D\right\Vert =0$
and $\left\Vert e_{1},D\right\Vert =4$. Then $\left\Vert b_{j},e_{1}\right\Vert =2$
for some $j\in[3]$. If $j=1$ then $b_{1}e_{1}b_{1}$ and $a_{2}b_{2}b_{3}a_{2}$
 are disjoint cycles. Else, say $j=2$. 
 By inspection, 
all low vertices are contained in $\{a_{1},b_{1},b_{3}\}$. If $b_{1}$
and $b_{3}$ are high then $b_{2}e_{1}b_{2}$, $b_{1}b_{3}e_{2}$
 beats $D,R$ by (P1). Else there is a $3$-cycle $D'\subseteq G[D+a_{1}]$
that contains every low vertex of $G$. Pick $D'$ with $b_{1}\in D'$ if
possible. If $b_{2}\notin D'$ then $D'$ and $b_{2}a'_{1}a'_{2}a_{2}b_{2}$
 are disjoint cycles. If $b_{3}\notin D'$ then $D'$, $b_{3}a_{2}a'_{2}a'_{1}$
 beats $D,R$ by (P1). Else $b_{1}\notin D'$, $a_{1}b_{1}\notin E$,
and $b_{1}$ is high. If $b_{1}a'_{1}\in E$ then $D'$, $b_{1}a'_{1}a'_{2}a_{2}$
 beats $D,R$ by (P1). Else, $\left\Vert b_{1},C\right\Vert =3$.
Then $D'$, $b_{1}c_{1}c_{2}b_{1}$, and $c_{3}e_{2}c_{3}$ are 
disjoint cycles.\end{proof}

\subsection{Key Lemma\label{sub:Key-L}}

Now $|R|=3$; say $R=a_{1}a'a_{2}$. By Lemma \ref{R=00003D3} the maximum length of a cycle in $\mathcal C$ is $4$.
Fix $C=w_{1}\dots w_{4}w_{1}\in\mathcal{C}$.

\begin{lem}
\label{L0} If $D\in\mathcal{C}$ with $\left\Vert R,D\right\Vert \geq7$
then $|D|=3$, $\left\Vert R,D\right\Vert =7$ and $G[R\cup D]=K_{6}-E(K_{3})$. \end{lem}
\begin{proof}
Since $\left\Vert R,D\right\Vert \geq7$, there exists $a\in R$ with
$\left\Vert a,D\right\Vert \geq3$.  By Claim~\ref{clm:1}, $|D|=3$.
If $\|a_i,D\|=3$ for any $i \in [2]$, then (a) and (c) in
Claim~\ref{PC} hold, violating \ref{o1}.
Then $\left\Vert a_{1},D\right\Vert =\left\Vert a_{2},D\right\Vert =2$ and
$\left\Vert a',D\right\Vert =3$. If $G[R\cup D]\ne K_{6}-K_{3}$
then $N(a_{1})\cap D\ne N(a_{2})\cap D$.
Then there is $w\in N(a_{1})\cap D$ with $\left\Vert a_{2},D-w\right\Vert =2$.
Then $wa_{1}a'w$ and $a_{2}(D-w)a_{2}$ are disjoint cycles.
\end{proof}

\begin{lem}
\label{L1} Let $D\in\mathcal{C}$ with $D=z_{1}\dots z_{t}z_{1}$.
If $\left\Vert C,D\right\Vert \geq8$ then $\left\Vert C,D\right\Vert =8$
and 
\[
W:=G[C\cup D]\in\{K_{4,4},~~K_{1}\vee K_{3,3},~~\overline{K}_{3}\vee(K_{1}+K_{3})\}.
\]
\end{lem}
\begin{proof}
 First suppose $|D|=4$. Suppose
 
\begin{equation}\label{star_p_14}\mbox{$W$ contains two disjoint
cycles $T$ and $C'$ with $|T|=3$.}
\end{equation}
Then $\mathcal{C}':=\mathcal{C}-C-D+T+C'$
is an optimal choice of $k-1$ disjoint cycles, since $\mathcal{C}$ is optimal. By Lemma~\ref{R=00003D3}, $|C'|\leq4$.
Thus $\mathcal{C}'$ beats $\mathcal{C}$ by \ref{o2}. 

\noindent CASE 1: $\Delta(W)=6$. By symmetry, assume $d_W(w_{4})=6$. Then $\left\Vert \{z_{i},z_{i+1}\},C-w_{4}\right\Vert \geq2$
for some $i\in\{1,3\}$. Then \eqref{star_p_14} holds with $T=w_{4}z_{4-i}z_{5-i}w_{4}$. 

\noindent CASE 2: $\Delta(W)=5.$ Say $z_{1},z_{2},z_{3}\in N(w_{1})$.
Then $\left\Vert \{z_{i},z_{4}\},C-w_{1}\right\Vert \geq2$ for some
$i\in\{1,3\}$. Then \eqref{star_p_14} holds with $T=w_{1}z_{4-i}z_{2}w_{1}$. 

\noindent CASE 3: $\Delta(W)=4$. Then $W$ is regular. If $W$ has
a triangle then \eqref{star_p_14} holds. Else, say $w_{1}z_{1},w_1z_3\in E$. 
Then $z_1,z_3 \not\in N(w_2) \cup N(w_4)$, so $z_2,z_4 \in N(w_2) \cup N(w_4)$, and $z_1,z_3 \in N(w_3)$.

Now, suppose $|D|=3$. 

\noindent CASE 1: $d_W(z_{h})=6$ for some $h\in[3]$. Say $h=3$. If
$w_{i},w_{i+1}\in N(z_{j})$ for some $i\in[4]$ and $j\in[2]$, then
$z_{3}w_{i+2}w_{i+3}z_{3}$, $z_{j}w_{i}w_{i+1}z_{j}$ beats $C,D$ by \ref{o2}. Else for all $j\in[2]$, $\|z_{j},C\|=2$, and
the neighbors of $z_j$ in $C$ are nonadjacent. If $w_i \in N(z_1)\cap N(z_2) \cap C$, then 
$z_3w_{i+1}w_{i+2}z_3$, $z_1z_2w_iz_1$ are preferable to $C,D$ by \ref{o2}.
Wence $W=K_{1}\vee K_{3,3}.$

\noindent CASE 2:  $d_W(z_{h})\le 5$ for every $h \in [3]$. Say $d(z_{1})=5=d(z_{2})$, $d(z_{3})=4$,
and $w_{1},w_{2},w_{3}\in N(z_{1})$. If $N(z_{1})\cap C\ne N(z_{2})\cap C$
then $W-z_{3}$ contains two disjoint cycles, preferable to $C,D$ by \ref{o2}; if $w_{i}\in N(z_{3})$
for some $i\in\{1,3\}$ then $W-w_{4}$ contains two disjoint cycles.
Then $N(z_{3})=\{w_{2},w_{4}\}$, and so $W=\overline{K}_{3}\vee(K_{1}+K_{3})$, where
$V(K_1)=\{w_4\}$, $w_2z_1z_2w_2=K_3$, and $V(K_3)=\{w_1,w_3,z_3\}$.
\end{proof}

\begin{claim}
\label{6edges}For $D \in \mathcal C$, if $\left\Vert \{w_{1},w_{3}\},D\right\Vert \geq5$
then $\left\Vert C,D\right\Vert \leq6$. If also $|D|=3$ then $\left\Vert \{w_{2},w_{4}\},D\right\Vert =0$.\end{claim}
\begin{proof}
Assume not.  Let $D=z_1\dots z_tz_1$. Then $\left\Vert \{w_{1},w_{3}\},D\right\Vert \geq5$ and  $\left\Vert C,D\right\Vert \geq7$.
Say $\left\Vert w_{1},D\right\Vert \geq\left\Vert w_{3},D\right\Vert $,
$\{z_{1},z_{2},z_{3}\}\subseteq N(w_{1})$, and $z_{l}\in N(w_{3})$. 

Suppose $\left\Vert w_{1},D\right\Vert =4$. Then $|D|=4$.
 If $\left\Vert z_{h},C\right\Vert \geq3$
for some $h\in[4]$ then 
there is a cycle $B\subseteq G[w_{2},w_{3},w_{4},z_{h}]$; so 
$B$, $w_1z_{h+1}z_{h+2}w_1$ beats $C,D$ by \ref{o2}.
Else there are  $j\in\{l-1,l+1\}$ and $i\in \{2,3,4\}$ with $z_iw_j\in E$. 
Then 
$z_lz_j[w_iw_3]z_l$, $w_1(D-z_l-z_j)w_1$ beats $C,D$  by \ref{o2}, where $[w_iw_3]=w_3$ if $i=3$.


Else, $\left\Vert w_{1},D\right\Vert =3$. By assumption, there is
$i\in\{2,4\}$ with $\left\Vert w_{i},D\right\Vert \geq1$. 
If $|D|=3$, applying Claim~\ref{PC} 
 with $P:=w_{1}w_{i}w_{3}$ and cycle $D$ yields two disjoint cycles in $(D \cup C)-w_{6-i}$, contradicting \ref{o2}. Therefore, suppose $|D|=4$. Because $w_1z_1z_2w_1$ and $w_1z_2z_3w_1$ are triangles, there do not exist cycles in $G[\{w_i,w_3,z_3,z_4\}]$ or $G[\{w_i,w_3,z_1,z_4\}]$ by~\ref{o2}. Then $\|\{w_i,w_3\},\{z_3,z_4\}\|$, $\|\{w_i,w_3\},\{z_1,z_4\}\| \leq 1$. Since $\|\{w_i,w_3\},D\| \geq 3$, one has a neighbor in $z_2$. If both are adjacent to $z_2$, then $w_iw_3z_2w_i$, $w_1z_1z_4z_3w_1$ beat $C,D$ by \ref{o2}. Then $\|\{w_i,w_3\},z_2\|=1=\|\{w_i,w_3\},z_1\|=\|\{w_i,w_3\},z_3\|$. Let $z_m$ be the neighbor of $w_i$. Then $w_iw_1z_mw_i$, $w_3(D-z_m)w_3$ beat $C,D$ by \ref{o2}.
 
 Suppose $|D|=3$ and $\|\{w_1,w_3\},D\| \geq 5$. If $\|\{w_2,w_4\},D\| \geq 1$, then $C \cup D$ contains two triangles, and these are preferable to $C,D$ by \ref{o2}.
\end{proof}

For $v\in N(C)$, set $\mbox{type}(v)=i\in[2]$ if $N(v)\cap C\subseteq\{w_{i},w_{i+2}\}$. Call
$v$  \emph{light }if $\left\Vert v,C\right\Vert =1$; else  $v$ is
 \emph{heavy.} For $D=z_{1}\dots z_{t}z_{1}\in\mathcal{C}$, put  $H:=H(D):=G[R\cup D]$.

\begin{claim}
\label{types} If $\left\Vert \{a_{1},a_{2}\},D\right\Vert \geq5$
then there exists $i\in[2]$ such that
\begin{enumerate}[label=(\alph{enumi}), ref=\alph{enumi}]
\item \label{4edges} $\|C,H\| \leq 12$ and $\left\Vert \{w_{i},w_{i+2}\},H\right\Vert \leq4$;
\item \label{12edges}$\left\Vert C,H\right\Vert =12$;
\item \label{pic}$N(w_{i})\cap H=N(w_{i+2})\cap H=\{a_{1},a_{2}\}$ and
$N(w_{3-i})\cap H=N(w_{5-i})\cap H=V(D)\cup\{a'\}$.\end{enumerate}
\begin{proof}
By Claim~\ref{clm:1}, $|D|=3$. Choose notation so that $\left\Vert a_{1},D\right\Vert =3$
and $z_{2},z_{3}\in N(a_{2})$. 

(\ref{4edges}) Using that $\{w_{1},w_{3}\}$ and $\{w_{2},w_{4}\}$
are independent and Lemma \ref{L1}:
\begin{equation}
\left\Vert C,H\right\Vert =\left\Vert C,V-(V-H)\right\Vert \geq2(4k-3)-8(k-2)=10.\label{eq:CH}
\end{equation}
Let $v\in V(H)$. As $K_4 \subseteq H$, $H-v$ contains a $3$-cycle. If $C+v$ contains
another $3$-cycle then these $3$-cycles beat $C,D$ by
\ref{o2}.
Thus, $\mbox{type}(v)$ is defined for all $v\in N(C)\cap H$,
 and $\left\Vert C,H\right\Vert \leq12$.
If only five vertices of $H$ have neighbors in $C$ then there is
$i\in[2]$ such that at most two vertices in $H$ have type $i$. Then $\left\Vert \{w_{i},w_{i+2}\},H\right\Vert \leq4$.
Else every vertex in $H$ has a neighbor in $C$.\emph{ }By (\ref{eq:CH}),
$H$ has at least four heavy vertices.

Let $H'$ be the spanning subgraph of $H$ with $xy\in E(H')$ iff $xy \in E(H)$ and
$H-\{x,y\}$ contains a $3$-cycle. If $xy\in E(H')$ then $N(x)\cap N(y)\cap C=\emptyset$ by \ref{o2}.
Now, if $x$ and $y$ have the same type, then they are both light. By inspection,
$H'\supseteq z_{1}a_{1}a'a_{2}z_{2}+a_{2}z_{3}$. 

Let $\mbox{type}(a_{2})=i$. If $a_{2}$ is heavy then its neighbors
$a',z_{2},z_{3}$ have type $3-i$. Either $z_{1},a_{1}$ are both
light or they have different types. Anyway, $\left\Vert \{w_{i},w_{i+2}\},H\right\Vert \leq4$.
Else $a_{2}$ is light. Then because there are at least four heavy vertices in $H$, 
at least one of $z_{1},a_{1}$ is heavy 
and so they have different types. Also any type-$i$ vertex in $a',z_{2},z_{3}$
is light, but at most one vertex of $a,z_2,z_3$ is light because there are at most two light vertices in $H$. Then $\left\Vert \{w_{i},w_{i+2}\},H\right\Vert \leq4$.

(\ref{12edges}) By (\ref{4edges}), there is $i$ with $\left\Vert \{w_{i},w_{i+2}\},H\right\Vert \leq4$;
thus 
\[
\left\Vert \{w_{i},w_{i+2}\},V-H\right\Vert \geq(4k-3)-4=4(k-2)+1.
\]
Now $\left\Vert \{w_{i},w_{i+2}\},D'\right\Vert \geq5$ for some $D'\in\mathcal{C}-C-D$.
By (\ref{4edges}), Claim~\ref{6edges}, and Lemma~\ref{L1},
\[
12\geq\left\Vert C,H\right\Vert =\left\Vert C,V-D'-(V-H-D')\right\Vert \geq2(4k-3)-6-8(k-3)=12.
\]

(\ref{pic}) By (\ref{12edges}), $\left\Vert C,H\right\Vert =12$,
so each vertex in $H$ is heavy. Thus $\mbox{type}(v)$ is the unique
proper $2$-coloring of $H'$, and (\ref{pic}) follows.
\end{proof}
\end{claim}
\begin{lem}
\label{lem6}\label{u_1u_3}There exists $C^{*}\in\mathcal{C}$ such
that $3\leq\left\Vert \{a_{1},a_{2}\},C^{*}\right\Vert \leq4$ and $\left\Vert \{a_{1},a_{2}\},D\right\Vert =4$
 for all $D\in\mathcal{C}-C^{*}$. If $\left\Vert \{a_{1},a_{2}\},C^{*}\right\Vert=3$ then one of $a_1,a_2$ is low. \end{lem}
\begin{proof}

Suppose $\left\Vert \{a_{1},a_{2}\},D\right\Vert \geq5$ for some  $D\in\mathcal C$; set $H:=H(D)$. Using Claim
\ref{types}, choose notation so that $\left\Vert \{w_{1},w_{3}\},H\right\Vert \leq4$. 
Now 
\[
\left\Vert \{w_{1},w_{3}\},V-H\right\Vert \geq4k-3-4=4(k-2)+1.
\]
Thus there is a cycle $B\in\mathcal{C}-D$ with $\left\Vert \{w_{1},w_{3}\},B\right\Vert \geq5$;
say $\left\Vert \{w_{1},B\}\right\Vert =3$. By Claim \ref{6edges},
$\left\Vert C,B\right\Vert \leq6$. 
Note by Claim~\ref{types}, if $|B|=4$ then for an edge $z_1z_2 \in N(w_1)$, $w_1z_1z_2w_1$ and $w_2w_3a_2a'w_2$ beat $B,C$ by \ref{o2}. Then $|B|=3$.
Using Claim~\ref{types}(\ref{12edges}) and Lemma~\ref{L1}, 
\[
2(4k-3)\leq\left\Vert C,V\right\Vert =\left\Vert C,H\cup B\cup(V-H-B)\right\Vert \leq12+6+8(k-3)=2(4k-3).
\]
Thus, $\left\Vert C,D'\right\Vert =8$ for all $D'\in \mathcal C-C-D$. By Lemma~\ref{L1}, $\left\Vert \{w_{1},w_{3}\},D'\right\Vert =\left\Vert \{w_{2},w_{4}\},D'\right\Vert =4$.
By Claim~\ref{types}(\ref{pic}) and Claim \ref{6edges},
\begin{align*}
4k-3 & \leq\left\Vert \{w_{2},w_{4}\},H\cup B\cup(V-H-B)\right\Vert \leq8+1+4(k-3)=4k-3,
\end{align*}
and so $\left\Vert \{w_{2},w_{4}\},B\right\Vert =1$. Say $\|w_2,B\|=1$. Since $|B|=3$, by Claim~\ref{PC}, $G[B\cup C-w_4]$ has two disjoint cycles that are  preferable to $C,B$ by \ref{o2}. This contradiction implies $\left\Vert \{a_{1},a_{2}\},D\right\Vert \leq4$ for all $D\in\mathcal{C}$. Since $\left\Vert \{a_{1},a_{2}\},V\right\Vert \geq{4k-3}$
and $\left\Vert \{a_{1},a_{2}\},R\right\Vert =2$, 
we get
$\left\Vert \{a_{1},a_{2}\},D\right\Vert \geq3$,
and equality holds for at most one $D\in\mathcal{C}$, and only if one of $a_1$ and $a_2$ is low.
\end{proof}

\subsection{Completion of the proof of Theorem \ref{main}.\label{contradiction}}

For an optimal $\mathcal{C}$, let $\mathcal{C}_{i}:=\{D\in\mathcal{C}:|D|=i\}$
and $t_{i}:=|\mathcal{C}_{i}|$. 
For $C\in\mathcal{C}_{4}$, let $Q_{C}:=Q_{C}(\mathcal{C}):=G[R(\mathcal{C})\cup C]$.
A $3$-path $R'$ is $\mathcal{D}$\emph{-useful} if $R'=R(\mathcal{C}')$
for an optimal set $\mathcal{C}'$ with $\mathcal{D}\subseteq\mathcal{C}'$;
we write \emph{$D$-useful} for $\{D\}$-useful. 

\begin{lem}
\label{2cases} \emph{Let} \emph{$\mathcal{C}$ be an optimal set
and $C\in\mathcal{C}_{4}$. Then }$Q=Q_{C}\in\{K_{3,4},K_{3,4}-e\}$. \end{lem}
\begin{proof}
Since $\mathcal{C}$ is optimal, $Q$ does not contain a $3$-cycle.
Then for all $v\in V(C)$, $N(v)\cap R$ is independent and $\left\Vert a_{1},C\right\Vert ,\left\Vert a_{2},C\right\Vert \leq2$.
By Lemma~\ref{lem6}, $\left\Vert \{a_{1},a_{2}\},C\right\Vert \geq3$.
Say $a_{1}w_{1},a_{1}w_{3}\in E$ and $\left\Vert a_{2},C\right\Vert \geq1$.
Then type($a_{1}$) and type($a_{2}$) are defined.
\begin{claim}
\label{LCtype}$\mbox{type}(a_{1})=\mbox{type}(a_{2})$.\end{claim}
\begin{proof}
Suppose not. Then $\left\Vert w_{i},R\right\Vert \leq1$ for all ${i}\in [4]$.
Say $a_{2}w_{2}\in E$. If $w_{i}a_{j}\in E$ and $\|a_{3-j},C\|=2$, let $R_{i}=w_{i}a_{j}a'$
and $C_{i}=a_{3-j}(C-w_{i})a_{3-j}$ (see Figure~\ref{R1C1}). Then $R_{i}$
is $(\mathcal C-C+C_{i})$-useful. Let $\lambda(X)$ be the number of low vertices
in $X\subseteq V$. As $Q$ does not contain a $3$-cycle, $\lambda(R)+\lambda(C)\leq2$.
We claim:
\begin{equation}
\forall D\in\mathcal{C}-C,~~\left\Vert a',D\right\Vert \leq2.\label{a'low}
\end{equation}
Fix $D\in\mathcal{C}-C$, and suppose
$\left\Vert a',D\right\Vert \geq3$. By Claim~\ref{clm:1}, $|D|=3$. Since

\noindent 
\begin{eqnarray}
\left\Vert C,D\right\Vert  & = & \left\Vert C,\mathcal C\right\Vert -\left\Vert C,\mathcal{C}-D\right\Vert \nonumber \\
 & \geq & 4(2k-1)-\lambda(C)-\left\Vert C,R\right\Vert -8(k-2)\notag\\
 & = & 12-\left\Vert C,R\right\Vert -\lambda(C)\geq6+\lambda(R),\label{CD}
\label{CD7}\end{eqnarray}
we get that  $\left\Vert w_{i},D\right\Vert \geq2$ for some $i\in[4]$.
 If $R_{i}$ is defined, $R_{i}$ is $\{C_{i},D\}$-useful.
 By Lemma~\ref{lem6},
$\left\Vert \{w_{i},a'\},D\right\Vert \leq4$. As $\left\Vert w_{i},D\right\Vert \geq2$,
$\left\Vert a',D\right\Vert \leq2$, proving (\ref{a'low}). 
Then $R_i$ is not defined, so $a_2$ is low with $N(a_2) \cap C = \{w_2\}$ and $\|w_2,D\|\leq 1$.
Then by (\ref{CD7}), $\|C-w_2,D\| \geq 6$. Note $G[a'+D]=K_4$, so for any $z \in D$, $D-z+a'$ is a triangle, so by \ref{o2} the neighbors of $z$ in $C$ are independent. Then $\|C-w_2,D\| = 6$ with $N(z) \cap C = \{w_1,w_3\}$ for every $z \in D$. Then $\|w_2,D\|=1$, say $zw_2 \in E(G)$, and now $w_2w_3zw_2$, $w_1(D-z)w_1$ beat $C,D$ by \ref{o2}.

\begin{figure}[ht]
 \begin{tikzpicture}[scale=0.75]
  \draw (-1,0) node[vertex,
  	 label=below:$w_1$,draw](w1){};
 \draw (-1,2) node[vertex,
  	 label=$w_2$,draw](w2){};
 \draw (1,2) node[vertex,
  	 label=$w_3$,draw](w3){};	
 \draw (1,0) node[vertex,
  	 label=below:$w_4$,draw](w4){};
\draw (w1) -- (w2) -- (w3) -- (w4) -- (w1);
 \draw (-3,-2) node[vertex,
  	 label=below:$a_1$,draw](a1){};
 \draw (3,-2) node[vertex,
  	 label=below:$a_2$,draw](a2){};
 \draw (0,-2.5) node[vertex,draw](a'){};
\draw (a1) -- (a') -- (a2);
\draw (w1) -- (a1);
\draw (w4) -- (a2);
\draw (a1) to[out=10,in=260] (w3);
\draw (a2) to[out=170,in=280] (w2);
\draw[dotted, line width=3pt] (w1) -- (a1) -- (a');
\draw[dotted, line width=3pt] (a2) -- (w4) -- (w3) -- (w2);
\draw[dotted, line width=3pt] (a2) to[out=170,in=280] (w2);
\draw (-3,-1) node {$R_1$};
\draw (2,1) node {$C_1$};
\end{tikzpicture}
\caption{Claim \ref{LCtype}}\label{R1C1}
\end{figure}
%

If $\|a',C\|\ge1$ then
 $a'w_{4}\in E$ and $N(a_{2})\cap C=\{w_{2}\}$.
Now $R_2$ is $C_2$-useful,
$\mbox{type}(a')\ne\mbox{type}(w_{2})$
with respect to $C_2$, and the middle vertex $a_2$ of $R_2$ has no neighbors
in $C_2$. Thus we may assume $\left\Vert a',C\right\Vert =0$. Then
$a'$ is low:
\begin{equation}
d(a')=\left\Vert a',C\cup R\right\Vert +\left\Vert a',\mathcal{C}-C\right\Vert \leq0+2+2(k-2)=2k-2.\label{a'}
\end{equation}
Thus all vertices of $C$ are high. Using Lemma~\ref{L1}, this yields:
\begin{equation}
4\geq\left\Vert C,R\right\Vert =\left\Vert C,V-(V-R)\right\Vert \geq4(2k-1)-8(k-1)=4.\label{8}
\end{equation}
As this calculation is tight, $d(w)=2k-1$ for every $w\in C$.
Thus $d(a')\ge 2k-2$, so (\ref{a'}) is tight. Hence $\left\Vert a',D\right\Vert =2$
for all $D\in\mathcal{C}-C$.

Pick $D=z_{1}\dots z_{t}z_{1}\in\mathcal{C}-C$ with $\left\Vert \{a_{1},a_{2}\},D\right\Vert $
maximum. By Lemma~\ref{lem6}, $3\leq\left\Vert \{a_{1},a_{2}\},D\right\Vert \leq4$.
$ $ Say $\left\Vert a_{i},D\right\Vert \geq 2$. By (\ref{8}), $\left\Vert C,D\right\Vert =8$.
By Lemma~\ref{L1}, 
\[
W:=G[C\cup D]\in\{K_{4,4},\,\,\overline{K}_{3}\vee(K_{3}+K_{1}),\,\, K_{1}\vee K_{3,3}\}.
\]

\noindent CASE 1: $W=K_{4,4}$. Then $\left\Vert D,R\right\Vert \geq5>|D|=4$,
so $\left\Vert z,R\right\Vert \geq2$ for some $z\in V(D)$. Let $w\in N(z)\cap C$.
Either $w$ and $z$ have a comon neighbor in $\{a_1,a_2\}$ or $z$ has
two consecutive neighbors in $R$. Regardless, $G[R+w+z]$ contains a
$3$-cycle $D'$ and $G[W-w-z]$ contains a $4$-cycle $C'$. Thus
$C',D'$ beats $C,D$ by \ref{o2}.

\noindent CASE 2: $W=\overline{K}_{3}\vee(K_{3}+K_{1})$. As $\left\Vert \{a',a_{i}\},D\right\Vert \ge 4>|D|$,
there is $z\in V(D)$ with $D':=za'a_{i}z\subseteq G$. Also $W-z$
contains a $3$-cycle $C'$, so $C',D'$ beats $C,D$
by \ref{o2}.

\noindent CASE 3: $W=K_{1}\vee K_{3,3}$. Some $v\in V(D)$ satisfies
$\left\Vert v,W\right\Vert =6$. There is no $w \in W-v$ such that
$w$ has two adjacent neighbors in $R$: else $a$
and $v$ would be contained in disjoint $3$-cycles, contradicting
the choice of $C,D$. Then $\left\Vert w,R\right\Vert \leq1$
for all $w\in W-v$, because type($a_1) \neq$ type($a_2$). Similarly, no $z\in D-v$ has two adjacent neighbors
in $R$. Thus
\[
2+3\leq\left\Vert a',D\right\Vert +\left\Vert \{a_{1},a_{2}\},D\right\Vert =\left\Vert R,D\right\Vert =\left\Vert R,D-v\right\Vert +\left\Vert R,v\right\Vert \leq2+3,
\]
so $\left\Vert \{a_{1},a_{2}\},D\right\Vert =3$, $R\subseteq N(v)$,
and  $N(a_{i})\cap K_{3,3}$  is independent.
 By Lemma~\ref{lem6} and the maximality of $\left\Vert \{a_{1},a_{2}\},D\right\Vert =3$,
$k=3$. Thus $G=\mathbf{Y}_{2}$, a contradiction.\end{proof}

Returning to the proof of Lemma~\ref{2cases}, we have $\mbox{type}(a_{1})=\mbox{type}(a_{2})$.
Using Lemma~\ref{lem6}, choose notation so that $a_{1}w_{1},a_{1}w_{3},a_{2}w_{1}\in E$.
Then $Q$ has bipartition $\{X,Y\}$ with $X:=\{a',w_{1},w_{3}\}$
and $Y:=\{a_{1},a_{2},w_{2},w_{4}\}$. The only possible nonedges
between $X$ and $Y$ are $a'w_{2}$, $a'w_{4}$ and $a_{2}w_{3}$.
Let $C':=w_{1}Rw_{1}$. Then $R':=w_{2}w_{3}w_{4}$ is $C'$-useful.
By Lemma~\ref{lem6}, $\left\Vert \{w_{2},w_{4}\},C'\right\Vert \geq3$.
Already $w_{2},w_{4}\in N(w_{1})$; so because $Q$ has no $C_3$, (say) $a'w_{2}\in E$.$ $
Now, let $C'':=a_{1}a'w_{2}w_{3}a_{1}$. Then $R'':=a_{2}w_{1}w_{4}$
is $C''$-useful; so $\left\Vert \{a_{2},w_{4}\},C''\right\Vert \geq3$.
Again, $Q$ contains no $C_3$, so $a'w_{4}$ or $a_{2}w_{3}$ is an edge of $G$. Thus
$Q\in\{K_{3,4},K_{3,4}-e\}$. 
\end{proof}

\begin{proof}[Proof of Theorem~\ref{main}]

 Using Lemma~\ref{2cases}, one of two cases holds:
\begin{enumerate}[label=(C\arabic*), ref=(C\arabic*)]
\item \noindent \label{C1} For some optimal
set $\mathcal{C}\mbox{ and }C'\in\mathcal{C}_{4},$ $Q_{C'}=K_{3,4}-x_{0}y_{0}$; 
\item \noindent \label{C2} for all optimal sets $\mathcal{C}\mbox{ and }C\in\mathcal{C}_{4}$, $G[R\cup C]=K_{3,4}$.
\end{enumerate}
\noindent Fix an optimal set $\mathcal{C}$ 
and $C'\in\mathcal{C}_{4}$,
where $R=y_0x'y$ with $d(y_0) \leq d(y)$,
such that in \ref{C1}, $Q_{C'}=K_{3,4}-x_0y_0$.
By
 Lemmas~\ref{lem6} and \ref{2cases},
for all $C\in\mathcal{C}_{4}$, $1\leq\left\Vert y_{0},C\right\Vert \leq\left\Vert y,C\right\Vert \leq2$
and $\left\Vert y_{0},C\right\Vert =1$
only in Case~\ref{C1} when $C=C'$. Put $H:=R\cup\bigcup\mathcal{C}_{4}$,  $S=S(\mathcal C):=N(y)\cap H$, and $T=T(\mathcal C):=V(H)\smallsetminus S$.
 As $\|y,R\|=1$  and $\|y,C\|=2$ for  each $C\in\mathcal C_4$,
 $|S|=1+2t_4=|T|-1$. 
\begin{claim}
\label{lem11}$H$ is a bipartite graph with parts $S$ and $T$. In case \ref{C1}, $H=K_{2t_{4}+1,2t_{4}+2}{-x_{0}y_{0}}$;
else $H=K_{2t_{4}+1,2t_{4}+2}$.\end{claim}
\begin{proof}
By Lemma \ref{2cases}, $\|x',S\|=\|y,T\|=\|y_0,T\|=0$.

By Lemmas~\ref{lem6} and \ref{2cases}, $\|y_0,S\|=|S|-1$ in \ref{C1} and $\|y_0,S\|=|S|$ otherwise. We claim that for every $t \in T - y_0$, $\|t,S\|=|S|$. This clearly holds for $y$, so take $t \in H-\{y,y_0\}$. Then $t \in C$ for some $C \in \mathcal C_4$. Let $\mathcal R^* := tx'y_0$ and $\mathcal C^*:=y(C-t)y$. (Note $R^*$ is a path and $C^*$ is a cycle by Lemma~\ref{2cases} and the choice of $y_0$.) Since $R^*$ is $C^*$-useful, by Lemmas~\ref{lem6} and \ref{2cases}, and by choice of $y_0$, $\|t,S\|=\|y,S\|=|S|$. Then in \ref{C1}, $H \supseteq K_{2t_{4}+1,2t_{4}+2}{-x_{0}y_{0}}$ and $x_0y_0 \not\in E(H)$;
else $H\supseteq K_{2t_{4}+1,2t_{4}+2}$.
  
  Now we easily see that if any edge exists inside $S$ or $T$, then $C_3+(t_4-1)C_4 \subseteq H$, and these cycles beat $\mathcal C_4$ by \ref{o2}.
\end{proof}

By Claim~\ref{lem11} all  pairs of vertices of $T$ are the ends of a $\mathcal C_3$-useful path. Now we use Lemma~\ref{lem6} to show that they have essentially the same degree to each cycle in $\mathcal C_3$.

\begin{claim}
\label{lem12}If $v\in T$ and $D\in\mathcal{C}_{3}$ then $1\leq\left\Vert v,D\right\Vert \leq2$;
if $\left\Vert v,D\right\Vert =1$ then $v$ is low and 
for all $C\in\mathcal{C}_{3}-D$, $\left\Vert v,C\right\Vert =2$.\end{claim}
\begin{proof}
By Claim~\ref{lem11}, $H+x_0y_0$ is a complete bipartite graph. Let $y_{1},y_{2}\in T-v$ and $u\in S-x_0$. Then
$R'=y_1uv$, $R''=y_2uv$, and $R'''=y_1uy_2$
are $\mathcal{C}_{3}$-useful. By Lemma~\ref{lem6},
$$3\leq\left\Vert \{v,y_{1}\},D\right\Vert,\left\Vert \{v,y_{2}\},D\right\Vert,\left\Vert \{y_{1},y_{2}\},D\right\Vert \leq4.$$
 Say $\left\Vert y_{1},D\right\Vert \leq2\leq\left\Vert y_{2},D\right\Vert $.
Thus 
\[
1\leq\left\Vert \{v,y_{1}\},D\right\Vert-\left\Vert y_{1},D\right\Vert =\left\Vert v,D\right\Vert =\left\Vert \{v,y_{2}\},D\right\Vert-\left\Vert y_{2},D\right\Vert \leq2.
\]
Suppose $\left\Vert v,D\right\Vert =1$. By Claim~\ref{lem11} and Lemma~\ref{lem6}, for any $v' \in T-v$,
\[
4k-3\leq\left\Vert \{v,v'\},H\cup(\mathcal{C}_{3}-D)\cup D\right\Vert \leq2(2t_{4}+1)+4(t_{3}-1)+3=4k-3.
\]
Thus for all $C\in\mathcal{C}_{3}-D_{0}$, $\left\Vert \{v,v'\},C\right\Vert =4$, and so $\left\Vert v,C\right\Vert =2$. Hence $v$ is low.
\end{proof}

Next we show that all vertices in $T$ have essentially the same neighborhood in each  $C\in \mathcal C_3$.

\begin{claim}\label{lem14-1}
Let $z\in D\in\mathcal{C}_{3}$ and $v,w\in T$ with $w$ high.
\begin{enumerate}
\item If $zv\in E$ and $zw\notin E$ then $T-w\subseteq N(z).$
\item $N(v)\cap D\subseteq N(w)\cap D$.
\end{enumerate}
\end{claim}
\begin{proof}
(1) Since $w$ is high, Claim~\ref{lem12} implies $\left\Vert w,D\right\Vert =2$.
Since $zw\notin E$, we see $D':=w(D-z)w$ is a $3$-cycle. Let $u\in S-x_0$.
 Then $zvu=R(\mathcal{C}')$ for some optimal set $\mathcal{C}'$
with $\mathcal{C}_{3}-D+D'\subseteq\mathcal{C}'$. By Claim~\ref{lem11},  $T(\mathcal{C}')=S+z$
and $S(\mathcal{C}')=T-w$.
If \ref{C2} holds, then $T-w=S(\mathcal C') \subseteq N(z)$, as desired. Suppose \ref{C1} holds, so 
 there are  $x_0 \in S$ and  $y_0 \in T$ with $x_0y_0\notin E$. 
 By Claims~\ref{lem11} and \ref{lem12}, $d(y_0) \leq (|S|-1)+2t_3=2k-2$, so $y_0$ is low. 
 Since $w$ is high, we see $y_0 \in T-w$. But now apply
 Claims~\ref{lem11} and \ref{lem12}  to $T(\mathcal C')$: $d(x_0) \leq |S(\mathcal C')|-1+2t_3 = 2k-2$, and $x_0$ is low. As $x_0y_0 \notin E$, this is a contradiction. 
 Now $T-w=S(\mathcal{C}')\subseteq N(z)$.

(2) Suppose there exists $z\in N(v)\cap D\smallsetminus N(w)$. By
(1), $T-w\subseteq N(z)$. Let $w'\in T-w$ be high. By Claim~\ref{lem12},
$\left\Vert w',D\right\Vert =2$. Now there exists $z'\in N(w)\cap D\smallsetminus N(w')$
and $z\ne z'$. By (1), $T-w'\subseteq N(z')$. As $|T|\geq4$ and
at least three of its vertices are high, there exists a high $w''\in T-w-w'$.
Since $w''z,w''z'\in E$, there exists $z''\in N(w)\cap D\smallsetminus N(w'')$
with $\{z,z',z''\}=V(D)$. By (1), $T-w''\subseteq N(z'')$. Since
$|T|\geq4$, there exists $x\in T\smallsetminus\{w,w',w''\}$. Now $\left\Vert x,D\right\Vert =3$,
contradicting Claim~\ref{lem12}. \end{proof}

Let $y_1,y_2 \in T-y_0$ and let $x \in S$ with $x=x_0$ if $x_0y_0 \not\in E$. By Claim~\ref{lem11},
$y_1xy_2$ is a path, and $G-\{y_1,y_2,x\}$ contains an optimal set $\mathcal C'$.
Recall $y_0$ was chosen in $T$ with minimum degree, so $y_1$ and $y_2$ are high and by Claim~\ref{lem12} $\|y_i,D\|=2$ for each $i \in [2]$ and each $D \in \mathcal C_3$. Let $N=N(y_{1})\cap\bigcup\mathcal{C}_{3}$ and $M=\bigcup\mathcal{C}{}_{3}\smallsetminus N$
(see Figure \ref{fig4}).
By Claim~\ref{lem11}, $T$ is independent. By Claim~\ref{lem14-1},
for every $y \in T$, $N(y) \cap \bigcup \mathcal C_3 \subseteq N$, so $E(M,T)=\emptyset$. Since $y_2 \neq y_0$, also $N(y_{2})\cap\bigcup\mathcal{C}_{3}=N$. 
\begin{figure}[ht]
 \begin{tikzpicture}[scale=0.6]
 \draw(0,0) node[vertex] (y1){} ;
 \draw(0.3,-.75) node {$y_1$};
 \draw(1,0) node[vertex] (y2){} ;
 \draw(1,-.5) node {$y_2$};
 \foreach \x in {1,...,4}
 \draw(1+\x,0) node[vertex] (t\x){} ;
 \draw(6.5,-1.5) node[label=$T$] (T){} ;
 \draw(2.5,-.3) ellipse (3cm and 1.2 cm);
 \draw(0.5,2.5) node[vertex,label=$x$,draw] (x'){} ;
 \foreach \x in {1,...,4}
 \draw(.5+\x,2.5) node[vertex] (s\x){} ;
 \draw(6.5,2.5) node[label=$S$] (S){} ;
 \draw(2.5,2.7) ellipse (3cm and 1 cm);
 \foreach \x in {1,...,4}
 	{\draw (s\x) -- (t\x);}
\draw (s1) -- (t2);
\draw (t1) -- (s2);
\draw (s3) -- (t4);
\draw (t3) -- (s4);
 \draw (y1) -- (x') -- (y2);
 \foreach \x in {1,2}
	 {\draw(-2.7*\x,2.05) node[vertex] (u\x){};
	 \draw(-2.7*\x-.7,0.2) node[vertex] (l\x){};
	 \draw(-2.7*\x+.7,0.2) node[vertex] (r\x) {}; }
 \foreach \x in {1,2}
	 { \draw (u\x) -- (l\x) -- (r\x) -- (u\x);} 
  \foreach \x in {1,2}
	  {\draw (l\x) to[out=-20, in=190] (y1) to[out=190, in=-20] (r\x);
 	\draw (l\x) to[out=-30, in=200] (y2) to[out=200, in=-30] (r\x);  }
  \draw(-4,0) ellipse (3.25 cm and 1 cm);
  \draw(-4,2.05) ellipse (3 cm and .5 cm);  
  \draw (-7.25,-.75) node {$N$};
  \draw (-7.25,2.75) node {$M$};
 \end{tikzpicture}\caption{}\label{fig4}
\end{figure}

\begin{claim}
\label{lemMM}$M$ is independent.\end{claim}
\begin{proof}
First, we show 
\begin{equation}\label{star_p_19}
\mbox{$\|z,S\|>t_4$ for all  $z\in M$.}
\end{equation}
 If not then there exists $z\in D\in\mathcal{C}_{3}$ with $\left\Vert z,S\right\Vert \leq t_{4}$.
Since $\left\Vert M,T\right\Vert =\|T,T\|=0$,
\[
\left\Vert \{y_{1},z\},\mathcal{C}_{3}\right\Vert \geq4k-3-\left\Vert \{z,y_{1}\},S\right\Vert \ge 4(t_{4}+t_{3}+1)-3-(2t_{4}+1+t_4)=t_{4}+4t_{3}>4t_{3}.
\]
Then there is $D'=z'z'_{1}z'_{2}z'\in\mathcal{C}_{3}$ with $\left\Vert \{z,y_{1}\},D'\right\Vert \geq5$
and $z'\in M$. As
$\|y_1,D\|=2$, 
 $\left\Vert z,D'\right\Vert =3$. Since $D^{*}:=zz'z'_{2}z$ 
is a cycle, $xy_{2}z'_{1}$ is $D^{*}$-useful.
As $\left\Vert z'_{1},D^{*}\right\Vert =3$, this contradicts Claim~\ref{lem12}, proving \eqref{star_p_19}.

Suppose $zz'\in E(M)$; say $z\in D\in\mathcal{C}_{3}$ and $z'\in D'\in\mathcal{C}_{3}$.
By \eqref{star_p_19}, there is $u\in N(z)\cap N(z')\cap S$. Then $zz'uz$, $y_{1}(D-z)y_{1}$ and
$y_{2}(D'-z')y_{2}$ are disjoint cycles, contrary to \ref{o1}.
\end{proof}

By Claims~\ref{lem11} and \ref{lemMM},  $M$ and $T$ are independent;  as remarked  above  $E(M,T)=\emptyset$. Then $M\cup T$ is independent. This contradicts (H3), since 
\[
|G|-2k+1=3t_{3}+4t_{4}+3-2(t_{3}+t_{4}+1)+1=t_{3}+2t_{4}+2=|M\cup T|\leq\alpha(G).
\]
 The proof of  Theorem~\ref{main} is now complete.
\end{proof}

\section{The case $k=2$}
Lov\'asz \cite{Lo} observed that any (simple or multi-) graph can be transformed into a multigraph with minimum degree at least $3$, without affecting the maximum number of disjoint cycles in the graph, by using a sequence of operations of the following three types: (i) deleting a bud; (ii) suppressing a vertex $v$ of degree $2$ that has two neighbors $x$ and $y$, i.e., deleting $v$ and adding a new (possibly parallel) edge between $x$  and $y$; and (iii) increasing the multiplicity of a loop or edge with multiplicity $2$. Here loops and two parallel edges are considered cycles, so forests have neither. Also $K_s$ and $K_{s,t}$ denote simple graphs. Let $W^*_s$ denote a wheel on $s$ vertices whose spokes, but not outer cycle edges, may be multiple. The following theorem characterizes those multigraphs that do not have two disjoint cycles.
\begin{thm}[Lov\'asz~\cite{Lo}]\label{lovasz}
Let $G$ be a multigraph  with $\delta(G) \geq 3$ and no two disjoint cycles. 
Then $G$ is one of the following: $(1)$ $K_5$, $(2)$ $W^*_s$,  $(3)$ 
$K_{3,|G|-3}$ together with a multigraph on the vertices of the (first) 3-class, and $(4)$ 
a forest $F$ and a vertex $x$ with possibly some loops at $x$ and some   
 edges linking $x$ to $F$.
\end{thm}

Let $\mathcal G$ be the class of simple graphs $G$ with $|G|\ge6$ and $\sigma_2(G)\geq 5$ that do not have two disjoint cycles. Fix $G\in\mathcal G$. A  vertex in $G$ is low if its degree is at most $2$.  The low vertices form a clique $Q$ of size at most $2$---if $|Q|=3$, then $Q$ is a component-cycle, and $G-Q$ has another cycle. By Lov\'asz's observation, $G$ can be reduced to a graph $H$ of type (1--4). Reversing this reduction, $G$ can be obtained from $H$ by adding buds and subdividing edges. 
Let $Q':=V(G)\smallsetminus V(H)$. It follows that  $Q\subseteq Q'$. If $Q' \neq Q$, then
$Q$ consists of a single leaf in $G$ with a neighbor of degree 3, so $G$ is obtained from $H$ by subdividing an edge and adding a leaf to the degree-2 vertex.
If $Q'=Q$, then $Q$ is a component of $G$, or $G=H+Q+e$ for some  edge $e\in E(H,Q)$, or at least one vertex of $Q$ subdivides an edge $e\in E(H)$. In the last case, when $|Q|=2$, $e$ is subdivided twice by $Q$. As $G$ is simple, $H$ has at most one multiple edge, and its multiplicity is at most $2$.

In case (4), because $\delta(H) \geq 3$,  either $F$ has at least two buds, each linked to $x$ by multiple edges, or $F$ has one bud linked to $x$ by an edge of multiplicity at least $3$. This case cannot arise from $G$. Also, $\delta(H)=3$, unless $H=K_5$, in which case $\delta(H)=4$. Then $Q$ is not an isolated vertex, lest deleting $Q$ leave $H$ with $\delta(H) \geq 5>4$; and if $Q$ has a vertex of degree $1$ then $H=K_5$. Else all vertices of $Q$ have degree $2$, and $Q$ consists of the subdivision vertices of one edge of $H$. 
We have the following lemma.

\begin{lem}\label{LoL}
Let $G$ be a graph with $|G|\geq 6$ and $\sigma_2(G) \geq 5$ that does not have  two disjoint cycles.
 Then $G$ is one of the following (see Figure~\ref{fig:LoL}):
\begin{enumerate}[label=(\alph*)]
\item $K_5+ K_2$;
\item $K_5$ with a  pendant edge, possibly subdivided;
\item $K_5$ with one edge subdivided and then a leaf added adjacent to the degree-2 vertex;
\item a graph $H$ of type (1--3) with no multiple edge, and possibly one edge subdivided once or twice, and if $|H|=6-i$ with $i\ge1$ then some edge is subdivided at least $i$ times; 
\item a graph $H$ of type (2) or (3) with one edge of multiplicity two, and one of its parallel parts is subdivided once or twice---twice if $|H|=4$.
\end{enumerate}
\end{lem}
\begin{figure}[ht]
\begin{center}
\begin{subfigure}[b]{0.2\textwidth}\centering
\begin{tikzpicture}[scale=0.3]
\foreach \y in {1,...,5}
	{\draw(-10,0) +(18+72*\y: 3 cm)  node[vertex] (k\y){} ;}
\foreach \y in {1,...,5}
\foreach \x in {1,...,5}
	\draw (k\x) -- (k\y);
\draw(-14,-2.4)node[vertex] (a){} ;
\draw(-14,.95)node[vertex] (b){} ;
\draw (a)--(b);
\end{tikzpicture}
\caption{}
\end{subfigure}\hspace{.75cm}\centering
\begin{subfigure}[b]{0.3\textwidth}
\begin{tikzpicture}[scale=0.3]
\foreach \y in {1,...,5}
	{\draw(-8.5,0) +(18+72*\y: 3 cm)  node[vertex] (k\y){} ;}
\foreach \y in {1,...,5}
\foreach \x in {1,...,5}
\draw (k\x) -- (k\y);
\draw(-6,3)  node[vertex] (u2){} ;
\draw  (u2) -- (k1);
\foreach \y in {1,...,5}
	{\draw(0,0) +(18+72*\y: 3 cm)  node[vertex] (k\y){} ;	}
\foreach \y in {1,...,5}
\foreach \x in {1,...,5}
\draw (k\x) -- (k\y);
\draw(2.5,3)  node[vertex] (u1){} ;
\draw(5,3)  node[vertex] (v1){} ;
\draw (v1) -- (u1) -- (k1);
\end{tikzpicture}
\caption{}
\end{subfigure}\hspace{1cm}
\begin{subfigure}[b]{0.2\textwidth}
\begin{tikzpicture}[scale=0.3]
\foreach \y in {1,...,3}
	{\draw(10,-4+2*\y)  node[vertex] (q\y){} ;}
\draw (14,-1) node[vertex] (p1){} ;
\draw (14,1) node[vertex] (p2){} ;
\draw (18,0) node[vertex] (u){} ;
\draw (21,0) node[vertex] (v){} ;
\draw (v) -- (u);
\foreach \x in {1,2}{
	 \draw (p\x) -- (u);
	 \foreach \y in {1,...,3} \draw (p\x) -- (q\y);
	 }
\draw (q1) -- (q2) -- (q3);
\draw (q1) to[out=180, in=180] (q3);
\end{tikzpicture}
\caption{}
\end{subfigure}
\vspace{.75cm}

\begin{subfigure}[b]{.8\textwidth}
\begin{tikzpicture}[scale=0.3]
\foreach \y in {1,...,5}
{
\draw(-5,-.75) +(18+72*\y: 3 cm)  node[vertex] (k\y){} ;
	}
\foreach \y in {1,...,5}
\foreach \x in {1,...,5}
\draw (k\x) -- (k\y);
\draw(-4,1.5) node[vertex] {} ;
\draw(-3.,.85) node[vertex] {} ;

\draw(4,-.75)  node[vertex] (mid){};
\foreach \y in {1,...,5}
{
\draw(4,-.75) +(18+72*\y: 3 cm)  node[vertex] (k\y){} ;
\draw (mid) -- (k\y);
	}
\draw (4,-.75) circle (3cm);
\draw(4,1.25) node[vertex] {} ;
\draw(4,.25)  node[vertex] {} ;
\draw(13,-.75)  node[vertex] (mid){};
\foreach \y in {1,...,5}
	{\draw(13,-.75) +(18+72*\y: 3 cm)  node[vertex] (k\y){} ;
	\draw (mid) -- (k\y);}
\draw (13,-.75) circle (3cm);
\draw(13,-.75) +(18+72*.33: 3 cm)  node[vertex] {} ;
\draw(13,-.75) +(18+72*.66: 3 cm)  node[vertex] {} ;
\draw(20,-3)  node[vertex] (x) {} ;
\draw(22,-3)  node[vertex] (y) {} ;
\draw(24,-3)  node[vertex] (z) {} ;
\draw[thick, densely dotted] (x) to[out=-45, in=-135] (z);
\draw[thick, densely dotted] (x)--(y) -- (z);
\draw(20,1.5)  node[vertex] (x') {} ;
\draw(22,1.5)  node[vertex] (z') {} ;
\draw(24,1.5)  node[vertex] (y') {} ;
\draw (x)--(x')--(y)--(y')--(z)--(z')--(y) (y')--(x)--(z') (x')--(z);
\draw(20,-1.5)  node[vertex]{} ;
\draw(20,0)  node[vertex]{} ;

\draw(28,-3)  node[vertex] (f) {} ;
\draw(30,-3)  node[vertex] (g) {} ;
\draw(32,-3)  node[vertex] (h) {} ;
\draw(28,1.5)  node[vertex] (f') {};
\draw(30,1.5)  node[vertex] (g') {};
\draw(32,1.5)  node[vertex] (h') {};
\draw (f)--(f')--(g)--(g')--(h)--(f') (g)--(f)--(g');
\draw (f)--(h')--(g) (h)--(h');
\draw(28.66,-3)  node[vertex]{} ;
\draw(29.33,-3)  node[vertex]{} ;
\draw[thick, densely dotted] (f) to[out=-45, in=-135] (h);
\draw[thick, densely dotted] (g) -- (h);
\end{tikzpicture}
\caption{}
\end{subfigure}

\vspace{.75cm}

\begin{subfigure}[b]{.6\textwidth}
\begin{tikzpicture}[scale=0.3]

\draw(5,-.75)  node[vertex] (mid){};
\foreach \y in {1,...,5}
{
\draw(5,-.75) +(18+72*\y: 3 cm)  node[vertex] (k\y){} ;
\draw (mid) -- (k\y);
	}
\draw (5,-.75) circle (3cm);
\draw(5,.25)  node[vertex]{};
\draw(5,1.25)  node[vertex]{};
\draw (5,-.75) to[out=30, in=-30] (5,2.25);

\draw(13,-.75)  node[vertex] (mid){};
\foreach \y in {1,...,5}
{
\draw(13,-.75) +(18+72*\y: 3 cm)  node[vertex] (k\y){} ;
\draw (mid) -- (k\y);
	}
\draw (13,-.75) circle (3cm);
\draw(13,-.75) +(18+72*.33: 3 cm)  node[vertex] {} ;
\draw(13,-.75) +(18+72*.66: 3 cm)  node[vertex] {} ;
\draw (k1)--(k5);

\draw(20,-3)  node[vertex] (x) {} ;
\draw(22,-3)  node[vertex] (y) {} ;
\draw(24,-3)  node[vertex] (z) {} ;
\draw(20,1.5)  node[vertex] (x') {} ;
\draw(22,1.5)  node[vertex] (z') {} ;
\draw(24,1.5)  node[vertex] (y') {} ;
\draw (x)--(x')--(y)--(y')--(z)--(z')--(y) (y')--(x)--(z') (x')--(z);
\draw(20,-1.5)  node[vertex]{} ;
\draw(20,0)  node[vertex]{} ;
\draw (x) to[out=135, in=-135] (x');

\draw(28,-3)  node[vertex] (f) {} ;
\draw(30,-3)  node[vertex] (g) {} ;
\draw(32,-3)  node[vertex] (h) {} ;
\draw(28,1.5)  node[vertex] (f') {};
\draw(30,1.5)  node[vertex] (g') {};
\draw(32,1.5)  node[vertex] (h') {};
\draw (f)--(f')--(g)--(g')--(h)--(f') (g)--(f)--(g');
\draw (f)--(h')--(g) (h)--(h');
\draw(28.66,-3)  node[vertex]{} ;
\draw(29.33,-3)  node[vertex]{} ;
\draw (f) to[out=-45, in=-135] (g);
\draw[thick, densely dotted] (f) to[out=-45, in=-135] (h);
\draw[thick, densely dotted] (g) -- (h);
\end{tikzpicture}\caption{}
\end{subfigure}
\end{center}
\caption{Theorem~\ref{LoL}}\label{fig:LoL}
\end{figure}

\section{Connections to Equitable Coloring\label{EqCol}}

 A proper vertex coloring of a graph $G$ is \emph{equitable} if any
two color classes differ in size by at most one. 
 In 1970
Hajnal and Szemer\' edi proved:
\begin{thm}[\cite{HSz}]
 \label{hst} 
  Every graph $G$ with $\Delta(G)+1\leq k$
has an equitable $k$-coloring. 
\end{thm}

 For  
a shorter proof of Theorem~\ref{hst}, see~\cite{KKY};
for an $O(k|G|^{2})$-time algorithm see~\cite{KKMS}.

Motivated by Brooks' Theorem, it is natural to ask which graphs $G$
with $\Delta(G)=k$ have equitable $k$-colorings. Certainly such
graphs are $k$-colorable. Also, if $k$ is odd then $K_{k,k}$ has
no equitable $k$-coloring. Chen, Lih, and Wu~\cite{CLW} conjectured
(in a different form) that these are the only obstructions
to an equitable version of Brooks' Theorem:
\begin{conjecture}[\cite{CLW}]
\label{conjmain} If $G$ is a graph with  $\chi(G), \Delta(G)\le k$ and no equitable
$k$-coloring then $k$ is odd and  $K_{k,k}\subseteq G$.
\end{conjecture}

In \cite{CLW}, Chen, Lih, and Wu proved Conjecture~\ref{conjmain} holds for $k=3$.
By a simple trick, it suffices to prove the conjecture for graphs $G$ with $|G|=ks$.
Combining the results of the two papers  \cite{KKJGT} and
 \cite{KK2}, 
we have:

\begin{thm}
\label{k,s<5}  Suppose $G$ is a graph with $|G|=ks$. If $\chi(G), \Delta(G)\le k$
and $G$ has no 
equitable $k$-coloring, then
$k$ is odd and $K_{k,k}\subseteq G$ or both 
 $k\ge5$ \cite{KKJGT} and
 $s\ge5$ \cite{KK2}.
\end{thm}

A graph $G$  is $k$-\emph{equitable} if $|G|=ks$, $\chi(G)\le k$ and every proper
$k$-coloring of $G$ has $s$ vertices in each color class. The following
strengthening of Conjecture~\ref{conjmain}, if true, provides a characterization of
graphs $G$ with $\chi(G), \Delta(G)\le k$ that have an equitable $k$-coloring.   

\begin{conjecture}[\cite{KK-Strong}]\label{strong} 
Every graph $G$ with  $\chi(G), \Delta(G)\le k$ has no equitable $k$-coloring if
and only if $k$ is odd and $G=H+K_{k,k}$ for some $k$-equitable graph $H$.
\end{conjecture}

The next theorem collects results from~\cite{KK-Strong}. Together with
Theorem~\ref{k,s<5} it yields Corollary~\ref{che+}.
\begin{thm}[\cite{KK-Strong}]
Conjecture~\ref{conjmain} is equivalent to Conjecture~\ref{strong}. Indeed, for any $k_0$ and $n_0$, Conjecture~\ref{conjmain} holds for  $k \leq k_0$ and $|G| \leq n_0$ if and only if 
Conjecture~\ref{strong} holds for  $k \leq k_0$ and $|G| \leq n_0$.  
\end{thm}

\begin{cor}
\label{che+}A graph $G$ with $|G|=3k$ and $\chi(G),\Delta(G)\leq k$
has no equitable $k$-coloring if and only if $k$ is odd and $G= K_{k,k} +K_k$. 
\end{cor}
We are now ready to complete our answer to Dirac's question for simple graphs.

\begin{proof}[Proof of Theorem~\ref{ch++}]
Assume $k\geq 2$ and $\delta(G)\geq2k-1$.
 It is apparent that
if any of (i), (H3), or (H4) in  Theorem~\ref{ch++}    fail, then $G$ does not have $k$ disjoint cycles. Now  
suppose $G$ satisfies (i), (H3), and (H4). If $k=2$ then $|G|\ge6$ and
$\delta(G)\ge3$. Thus $G$ has no subdivided edge, and only (d) of Lemma~\ref{LoL} is
possible. By (i), $G\ne K_5$; by (H4), $G$ is not a wheel; and by (H3),  $G$ is not
type (3) of Theorem~\ref{lovasz}. Then $G$ has $2$ disjoint cycles. Finally, suppose
$k\ge3$.
 Since $G$ satisfies (ii), we see $G\notin \{\mathbf Y_1,\mathbf Y_2\}$ and $G$ satisfies
(H2). If $|G|\ge3k+1$, then  $G$ has $k$ disjoint cycles by Theorem~\ref{main}.
Otherwise, $|G|=3k$ and $G$ has $k$ disjoint cycles if and only if its vertices can
be partitioned into disjoint $K_3$'s. This is equivalent to $\overline G$ having an
equitable $k$-coloring. By (ii), $\Delta(\overline G)\le k$, and by (H3), 
$\omega(\overline G)\le k$. Then
 by Brooks' Theorem, $\chi(\overline G)\le k$. By (H4) and Corollary~\ref{che+}, $\overline G$
has an equitable $k$-coloring. 
 \end{proof}

Next we turn to Ore-type results on equitable coloring. To complement
Theorem~\ref{main}, we need a theorem that characterizes when a graph $G$ with
$|G|=3k$  that satisfies (H2) and (H3) has $k$ disjoint cycles, or equivalently,
when its complement $\overline G$ has an equitable coloring. The complementary
version of $\sigma_2(G)$ is the  \emph{maximum Ore-degree} $\theta(H):=\max_{xy\in
E(H)}(d(x)+d(y))$. 
Then
$\theta(\overline{G})=2|G|-\sigma_{2}(G)-2$, and if $|G|=3k$ and 
$\sigma_{2}(G)\geq4k-3$ then $\theta(\overline{G})\leq2k+1$. Also, if $G$ satisfies
(H3) then $\omega(\overline G)\le k$.
This would correspond to an Ore-Brooks-type theorem on equitable coloring.

Several papers, including \cite{KK-Ore,KK-Ore2,KRS}, address equitable colorings of
graphs $G$ with $\theta(G)$ bounded from above. For instance, the following is a natural Ore-type
version of Theorem~\ref{hst}.

\begin{thm}[\cite{KK-Ore}]
 \label{hst-ore} Every graph $G$ with
$\theta(G)\leq2k-1$ has an equitable $k$-coloring. \end{thm}

Even for proper (not necessarily equitable) coloring, an Ore-Brooks-type theorem requires forbidding some
extra subgraphs when  $\theta$ is $3$ or $4$. It was observed in \cite{KK-Ore2} that for 
$k=3,4$ there are graphs for which $\theta(G)\le2k+1$ and $\omega(G)\le k$, but
$\chi(G)\ge k+1$. The following theorem was proved for $k\ge6$ in \cite{KK-Ore2} and
then for $k\ge5$ in \cite{KRS}.  

\begin{thm}
\label{chrom} Let $k\ge5$. If $\omega(G)\le k$ and $\theta(G)\le2k+1$, then
$\chi(G)\le k$.
\end{thm}

In the subsequent paper~\cite{KKMY}  we prove an analog of Theorem~\ref{main} for $3k$-vertex graphs.
\\

{\bf Acknowledgment.} We thank the referees for the helpful comments.

\end{document}